\documentclass[review,hidelinks,onefignum,onetabnum]{siamart251216}

\usepackage{lipsum}

\usepackage{amsfonts}
\usepackage{amssymb}
\usepackage{amsopn}

\usepackage{bm}

\usepackage{graphicx}
\usepackage{subcaption}
\usepackage{epstopdf}

\usepackage{algorithmic}
\usepackage{algorithm}

\usepackage{booktabs}
\usepackage{multirow}
\usepackage{array}

\usepackage{tabularx}
\usepackage{ragged2e}

\usepackage{xfrac}
\usepackage{textcomp}

\usepackage{tcolorbox}

\usepackage{makecell}

\ifpdf
  \DeclareGraphicsExtensions{.eps,.pdf,.png,.jpg}
\else
  \DeclareGraphicsExtensions{.eps}
\fi

\newsiamremark{remark}{Remark}

\newsiamremark{hypothesis}{Hypothesis}

\crefname{hypothesis}{Hypothesis}{Hypotheses}

\newsiamthm{claim}{Claim}

\newsiamremark{fact}{Fact}

\crefname{fact}{Fact}{Facts}

\newcolumntype{L}
{>{\RaggedRight\arraybackslash}X}

\allowdisplaybreaks 
\usepackage{booktabs}

\usepackage{textcomp}
\usepackage{xfrac}

\headers{An Example Article}{D. Doe, P. T. Frank, and J. E. Smith}

\title{An Example Article\thanks{Submitted to the editors DATE.
\funding{This work was funded by the Fog Research Institute under contract no.~FRI-454.}}}

\author{Dianne Doe\thanks{Imagination Corp., Chicago, IL 
  (\email{ddoe@imag.com}, \url{http://www.imag.com/\string~ddoe/}).}
\and Paul T. Frank\thanks{Department of Applied Mathematics, Fictional University, Boise, ID 
  (\email{ptfrank@fictional.edu}, \email{jesmith@fictional.edu}).}
\and Jane E. Smith\footnotemark[3]}

\headers
{Y.D. Gan and G. Wu}
{Kaczmarz-type Methods For Solving Doubly Noisy Linear Systems}

\title
{On Convergence Behavior of Randomized Kaczmarz-type Methods for Solving Doubly Noisy Linear Systems}

\author{
Yudan Gan
\thanks{
School of Mathematics,
China University of Mining and Technology,
Xuzhou, 221116, Jiangsu, P.R. China.
E-mail: {\tt ganyudan339@163.com}.
This work is supported by the Postgraduate Innovation Program of China University of Mining and Technology under Grant 2026WLKXJ237, and the Postgraduate Research Practice Innovation Program of Jiangsu Province under Grant 26CXJH4750.
}
\and
Gang Wu
\thanks{
Corresponding author.
School of Mathematics,
China University of Mining and Technology,
Xuzhou, 221116, Jiangsu, P.R. China.
E-mail: {\tt gangwu@cumt.edu.cn}.
This work is supported by the National Natural Science Foundation of China under Grants 12671452 and 12271518.
}
} 

\ifpdf
\hypersetup{
  pdftitle={An Example Article},
  pdfauthor={D. Doe, P. T. Frank, and J. E. Smith}
}
\fi

\begin{document}
\nolinenumbers

\maketitle

\begin{abstract}
In this paper, we consider the convergence of some popular Kaczmarz-type methods, including the randomized Kaczmarz (RK) method,
the randomized extended Kaczmarz (REK) method, the randomized block Kaczmarz (RBK) method, and the randomized double block Kaczmarz (RDBK) method for doubly noisy linear system. 
To the best of our knowledge, this is the first work that considers and compares the convergence of the Kaczmarz-type methods for consistent linear systems, as well as that for inconsistent linear systems to solve double noisy linear system. We prove that these methods converge to a neighborhood of the smallest norm least squares solution of the underlying noiseless system. Specifically, RK and RBK that are designed for consistent linear systems, often outperform REK and RDBK that are designed for inconsistent linear systems, respectively, in terms of the convergence rate or the convergence horizon. Moreover, we establish the convergence of the more general sketch-and-project framework for the doubly noisy linear system. The analysis provides more comprehensive
results without imposing any additional restrictions. Numerical experiments are performed to validate our theoretical results.
\end{abstract}



\section{Introduction}
With the rapid development of artificial intelligence and data science, the solution of large-scale matrix computation problems has attracted increasing attention. Consider the following large-scale linear system
		\begin{equation}\label{1.2}
		A \bm {x} = \bm {b},\quad A \in \mathbb{R}^{ m\times n},\quad \bm {b} \in \mathbb{R}^{ m}.
	\end{equation}
In this work, \eqref{1.2} can be {\it consistent {\rm or} inconsistent}. This problem has extensive applications including computed tomography \cite{Natter-2001}, medical image reconstruction \cite{Herman-1993}, signal processing \cite{Feich-1992,Xu-2025}, machine learning \cite{Bjorck-2024}, inverse problem \cite{Yin-2025, Tondji-2024, Yuan-2022}, and so on. 

With significant growth in the scale of the data sets, iterative methods are inevitable choices for settle this problem \cite{Bjorck-2015,saad-2003}, among which the Kaczmarz-type methods are efficient choices for solving both consistent \cite{Bai-2026,Bai-2018,Bai-2018-2,Wu-2023,Shao-2023,Yin-2023} and inconsistent large-scale linear systems \cite{Bai-2023,Bai-2021,Wu-2025,Yin-2025-2,Yin-2024}. Thanks to its memory-efficient, Kaczmarz method has been widely used in many practical applications such as image reconstruction, distributed computation, and so on \cite{Eamaz-2024,Freris-2012,Gordon-1970,vanLith-2021}.

The idea of the classical Kaczmarz-type method is to update the current iteration by projecting onto one row of the system \eqref{1.2} in each step \cite{Kacz-1937}.
More precisely, let $\bm a^\top_{i_k}$ be the ${i_k}$-th row of $A$ and let $b_{i_k}$ be the ${i_k}$-th entry of $\bm b$ at the $k$-th iteration, respectively. The classical Kaczmarz method orthogonally projects the current iteration $\bm {x}_{k-1}$ onto the affine hyperplane defined by one row $\bm a^\top_{i_k} \bm{x}=b_{i_k}$ of \eqref{1.2}, where the index set $i_k$ is chosen cyclically from the row indices $\{1,2,\ldots,m\}$. Starting from an initial point $\bm x_0$, the $k$-th updating has the following form \cite{Kacz-1937}
$$
\bm x_{k} = \bm x_{k-1} - \frac{ \bm {a}^\top_{i_k} \bm x_{k-1}- b_{i_k}}{\|\bm a_{i_k}\|_2^2} \bm a_{i_k}, \quad k=1,2,\ldots,
$$
where \(i_k=((k-1) \bmod m)+1\). Instead of selecting rows cyclically, Strohmer and Vershynin \cite{Strohmer-2009} proposed a randomized Kaczmarz method (RK) by randomly selecting rows with probability proportional to the square of their 2-norms. Moreover, it was shown that the RK method converges exponentially in expectation to the unique solution of consistent linear systems \eqref{1.2}. The efficiency of the RK algorithm has attracted extensive interest of scholars, and so it was further investigated 
\cite{Bai-2018-3,Su-2026,Tondji-2021,Wang-2022, Zhang-2019}. 

The randomized block Kaczmarz (RBK) method \cite{Need-2014} is a block version of the RK method. In this method, the row indices $\{1,2,\ldots,m\}$ are divided into $p$ subsets, denoted as a partition $\{\tau_1, \tau_2,\ldots,\tau_p\}$, and the corresponding row submatrix is denoted by ${A}_{\tau_{i_k}}$. At the $k$-th step, a block index $\tau_{i_k} \in \{\tau_1, \tau_2,\ldots,\tau_p\}$ is chosen randomly with probability \scalebox{0.8}{$\|A_{\tau_{i_k}}\big\|_F^2/\big\|A\big\|_F^2$}, and the updating has the following form
\begin{equation}\label{rbk}
\bm x_{k} = \bm x_{k-1} + {A}_{\tau_{i_k}}^\dagger (\bm{b}_{\tau_{i_k}}- A_{\tau_{i_k}} \bm x_{k-1}).
\end{equation}



Recently, considerable attention has been devoted to solving large-scale inconsistent linear systems by using Kaczmarz-type methods \cite{Du-2020,Wu-2025,Need-2010,Yin-2024,Zouzias-2013}. Among these methods, the randomized extended Kaczmarz (REK) method \cite{Zouzias-2013} is a widely used one that converges exponentially in mean square to the smallest norm least squares (LS) solution of a inconsistent linear system. More precisely, let $\bm x_0=\bm 0$ and $\bm z_0=\bm b$, in the $k$-th step of the REK method \cite{Zouzias-2013}, a row $\bm { a}^\top_{i_k}$ and a column $\bm{ a}_{j_k}$ are chosen from $A$ with probability
$\left\|\bm{a}_{i_{k}}\right\|_2^2 / \|A \|_F^2$ and $\left\| \bm{a}_{j_{k}}\right\|_2^2 / \|A \|_F^2$, respectively, and the $k$-th updating has the following form 
\begin{equation}\label{rek}
\begin{cases}
\bm z_{k}=\bm z_{k-1}-\frac{\bm a^\top_{j_k}
	\bm z_{k-1}}{\left\|\bm{a}_{j_k}\right\|_2^2}\bm{a}_{j_k},\\[6pt]
\bm x_{k}=\bm x_{k-1}- \frac{\bm a^\top_{i_k} \bm x_{k-1} - b_{i_k}+( z_{k-1})_{i_k}}{\left\|\bm {a}_{i_k}\right\|_2^2} \bm a_{i_k},
\end{cases}
\end{equation}
where $(z_{k-1})_{i_k}$ is the ${i_k}$-th entry of $\bm z_{k-1}$.



The randomized double block Kaczmarz (RDBK) method \cite{Need-2015} is developed to solve inconsistent linear systems, which can be viewed as a combination of the REK and the RBK methods. Let $\{\nu_1,\ldots,\nu_q\}$ be a partition of $\{1,2,\ldots,n\}$, and let $\nu_{i_k}$ be a block index selected from $\{\nu_1, \nu_2,\ldots,\nu_q\}$. Initialize $\bm{x}_0=\bm 0$ and $\bm{z}_0=\bm b$, then the $k$-th updating in the RDBK method \cite{Need-2015} is
\begin{equation}\label{rdbk}
	\begin{cases}
\bm{z}_k=\bm{z}_{k-1}-A_{\nu_{i_k}}A_{\nu_{i_k}}^\dagger\bm{z}_{k-1},\\[6pt]
\bm{x}_k= \bm{x}_{k-1} + {A}_{\tau_{i_k}}^\dagger (\bm{b}_{\tau_{i_k}}-(\bm{z}_k)_{\tau_{i_k}}-{A}_{\tau_{i_k}} \bm{ x}_{k-1}),
\end{cases}
\end{equation}
in which a row submatrix ${A}_{\tau_{i_k}}$ is randomly selected with probability $\|A_{\tau_{i_k}}\big\|_F^2/\big\|A\big\|_F^2$, a column submatrix $A_{\nu_{i_k}}$ is randomly selected with probability $\|A_{\nu_{i_k}}\big\|_F^2/\big\|A\big\|_F^2$, and $(\bm{z}_k)_{\tau_{i_k}}$ is a subvector of $\bm{z}_k$ indexed by $\tau_{i_k}$. Since then, many variants \cite{Du-2019-2,Du-2020,Nec-2019,Wu-2025-2} of the block Kaczmarz algorithm have been developed to accelerate convergence for solving \eqref{1.2}. 

In \cite{Gower2015}, Gower and Richt{\'a}rik developed a fundamental and simple randomized iterative method for solving \eqref{1.2}, which recovers some well-known randomized methods as special cases, including the randomized (block) Kaczmarz method, randomized Newton method, randomized coordinate descent method, and random Gaussian pursuit. This method has six different but equivalent interpretations, among which the sketch-and-project method draws a random
matrix and uses it to generate a new approximation:
\begin{equation}\label{SP}
\begin{aligned}
    \bm{x}_{k+1}
    =
    \bm{x}_{k}
    -
    B^{-1}A^{\top}S_k
    \left(S_k^{\top}AB^{-1}A^{\top}S_k\right)^{\dagger}
    S_k^{\top}\left(A\bm{x}_{k}-\bm{b}\right),
\end{aligned}
\end{equation}
where $\bm{x}_{k+1}$, $\bm{x}_{k}$ are the approximations from the current and previous steps, respectively, $S_k \in \mathbb{R}^{ m\times q}$ is the random sketch matrix drawn in an independent and identically distributed (i.i.d.) fashion at each step, and $B \in \mathbb{R}^{n \times n}$ is symmetric positive definite called the metric matrix. For more details and developments on the sketch-and-project method, one refers to \cite{Gower2021,Mori2018,Patel2021,Rich2020} and the references therein.

In practice, noise is inevitable in real-world data. In this work, we pay special attention to solving the {\it doubly noisy} linear system
\begin{equation}\label{1.1}
		\widetilde{A} \bm{x} =\bm{\tilde{b}},
\end{equation}
by the Kaczmarz-type methods, where $\widetilde{A}=A+E$ and $\bm {\tilde{b}}=\bm b+\bm f$ are the noisy versions of $A$ and $\bm b$, respectively. More precisely, we focus on quantifying how far the iterative solutions are from the smallest norm least squares (LS) solution $\bm {x}_{LS}=A^{\dag}{\bm b}$ of \eqref{1.2}, when Kaczmarz-type methods are directly applied to \eqref{1.1}. 
Notice that both the doubly noisy linear system \eqref{1.1} and the underlying noiseless linear system \eqref{1.2} are not required to be consistent at all. 

To the best of our knowledge, the analysis on the convergence of the Kaczmarz-type methods on the doubly noisy linear system \eqref{1.1} is limited. For the {\it single noisy} linear system 
\begin{equation}\label{1.3}
A \bm x = \bm{\tilde{b}},
\end{equation}
where $\bm{\tilde{b}}=\bm b+\bm \epsilon$, Needell \cite{Need-2010} proved that the RK method exponentially converges in expectation to a ball centered at $\bm {x}_{LS}$ under the assumption that $A$ has full column rank. The radius of this ball is referred to as the convergence horizon \cite{Need-2014}. Zouzias and Freris \cite{Zouzias-2013} dropped the full column rank assumption and further improved the convergence horizon for solving \eqref{1.3}. In \cite{Bergou-2024}, Bergou {\it et al.} analyzed the RK method for the doubly noisy linear system \eqref{1.1}. It is shown that the approximation $\bm {{\tilde x}}_{k}$ generated by the RK method for solving \eqref{1.1} will converge to a ball centered at $\bm {x}_{LS}$ under the assumption that $\bm {\tilde x}_{0}- \bm {x}_{LS} \in \mathcal{R}(\widetilde{A}^\top)$, where $\bm {\tilde x}_{0}$ is the initial guess and $\mathcal{R}(\widetilde{A}^\top)$ denotes the row space of $\widetilde{A}$. 
Moreover, the convergence horizon depends on the noise level of both $\widetilde{A}$ and $\bm{\tilde{b}}$ \cite{Bergou-2024}. Recently, Bergou {\it et al.} \cite{Bergou-2025} removed this assumption, and made use of the projection of  $\bm {x}_{LS}$ onto the row space of $\widetilde{A}$ as the limiting reference point. In \cite{Wu-n-2026}, Wu {\it et al.} gave a perturbation analysis on the sampling semi-randomized Kaczmarz (SSRK) method for solving \eqref{1.1}.


In practice, what we truly care about is the distance between the approximation $\bm {{\tilde x}}_{k}$ from the Kaczmarz-type methods
and the LS solution $\bm {x}_{LS}$ of the underlying noiseless linear system \eqref{1.2}. Therefore, it is necessary to revisit the convergence of the Kaczmarz-type methods on the doubly noisy linear system \eqref{1.1}. The main contributions of this work are summarized as follows:
\begin{itemize}
\item
Firstly, we revisit the convergence behavior of the RK method applied to the doubly noisy linear system \eqref{1.1}. 
Our result is free of the strict assumption $\widetilde{\bm{x}}_0-\bm{x}_{\mathrm{LS}}\in\mathcal{R}(\widetilde{A}^{\top})$ that is made in \cite[Theorem 3.1]{Bergou-2024}. 

\item
Secondly, from the norm of the expectation point of view, we consider the convergence of the RK and the REK methods for the doubly noisy linear system \eqref{1.1}. It is shown that the asymptotic convergence rate of the RK method is faster than that of REK, while the convergence horizon of REK coincides with that of RK. Thus,  the RK method will be better than the REK method when solving the doubly noisy linear system \eqref{1.1}.

\item
Thirdly, from the perspective of the norm of the expected error, we also investigate the convergence of the RBK and RDBK methods for the doubly noisy linear system \eqref{1.1}. We point out that the asymptotic convergence rates of RBK and RDBK are comparable, while the convergence horizon of RDBK is higher than that of RBK. Thus, the RBK method will be superior to the RDBK method for solving the doubly noisy linear system \eqref{1.1}.

\item
Finally, Kaczmarz-type methods are special cases of the sketch-and-project method \cite{Gower2015}. In this work, we establish the convergence of the sketch-and-project method for the doubly noisy linear system \eqref{1.1}, both from the perspectives of the expectation of the error norm and the norm of the expected error. Compared with the existing results, our analysis provides  more general results without imposing any additional restrictions.



\end{itemize}

Let's introduce some preliminary notations. 
Let $\mathbb{E}_{k-1}[\cdot]$ be the expectation conditional on the first $k-1$ iterations of the Kaczmarz-type algorithms. Then we have
\begin{align*}
	\mathbb{E}_{k-1}[\cdot]=\mathbb{E}[\cdot |j_1, i_1, j_2, i_2,\ldots,j_{k-1}, i_{k-1}],
\end{align*}
where $i_m$ is the $m$-th row chosen and $j_m$ is the $m$-th column chosen. The expectation conditional on both the first $k-1$ iterations and the column selected at the $k$-th iteration is defined as
$$
\mathbb{E}_{k-1}^i[\cdot]=\mathbb{E}[\cdot |j_1, i_1, j_2, i_2,\ldots,j_{k-1}, i_{k-1}, j_{k}].
$$
By the law of total expectation, it follows that
$\mathbb{E}_{k-1}[\cdot]=\mathbb{E}_{k-1}[\mathbb{E}_{k-1}^i[\cdot]].$ Other commonly used notations are summarized in Table \ref{tab:notations}.


\begin{table}[htbp]
	\centering
	\caption{Some notations used in this paper}
	\label{tab:notations}
	\begin{tabular}{ll}
		\toprule
		Notations & Description \\
		\midrule
		$M^\dagger, M^\top$ & Moore-Penrose inverse and transpose of a matrix $M$ \\
		$\sigma_{\min}(M)$ & the smallest nonzero singular value of a matrix $M$ \\
		$\sigma_{\max}(M)$ & the largest singular value of a matrix $M$ \\
		$\bm{\tilde{a}_{i_k}}^\top, \bm{\tilde{a}_{j_k}}$ & the $i$th row and $j$th column of $\widetilde{A}$ at the $k$th iterate\\
		$A\succeq B~(A\succ B)$ & $A-B$ is semi-positive definite (positive definite)\\
		$\lambda_{\min}(M),\lambda_{\max}(M)$ & the smallest and largest eigenvalue of $M$\\
			$\{\tau_1, \tau_2,\ldots,\tau_p\}$  & partitionings of the row indices $\{1,2,\ldots,m\}$ \\
		$\{\nu_1, \nu_2,\ldots,\nu_q\}$ & partitionings of the column indices $\{1,2,\ldots,n\}$\\
	$\widetilde{A}_{\tau_{i_k}}\in \mathbb{R}^{|\tau_{i_k}| \times n}$ & the row submatrix of $\widetilde{A}$ indexed by \scalebox{0.9}{$\tau_{i_k} \in \{\tau_1, \tau_2,\ldots,\tau_p\}$}\\
	$\widetilde{A}_{\nu_{i_k}}\in \mathbb{R}^{m \times |\nu_{i_k}|}$ & the column submatrix of $\widetilde{A}$ indexed by \scalebox{0.9}{$\nu_{i_k}\in \{\nu_1, \nu_2,\ldots,\nu_q\}$}\\
	$\mathcal{N}(\widetilde{A}), \mathcal{R}(\widetilde{A})$ & Null space and range of $\widetilde{A}$ \\
	$(\bm {x})_{\mathcal{N}(\widetilde{A})}, (\bm{x})_{\mathcal{R}(\widetilde{A})}$ & the orthogonal projection of $\bm {x}$ onto $\mathcal{N}(\widetilde{A})$ and  $\mathcal{R}(\widetilde{A})$ \\
	$\bm {x}_{LS}, \bm {\tilde x}_{LS}, \bm {\tilde x}^{(i_k)}_{LS}$ & the LS solutions of \eqref{1.2}, \eqref{1.1} and subsystem \scalebox{0.9}{$\widetilde{A}_{\tau_{i_k}} \bm x=\bm{\tilde{b}}_{\tau_{i_k}}$} \\
%
%
    
	\bottomrule
	\end{tabular}
\end{table}


The remainder of the paper is organized as follows. Section~\ref{sec2} reviews related work on Kaczmarz-type methods for solving the noisy linear systems \eqref{1.1} and \eqref{1.3}. In Section~\ref{sec4}, we revisit the convergence of the RK method applied to \eqref{1.1}, and provide a more detailed analysis on the behavior of
$\mathbb{E}\!\left[\left\|\widetilde{\bm{x}}_k-\bm{x}_{\mathrm{LS}}\right\|_2^2\right]$. Section~\ref{sec3} analyzes the upper bound for
$\left\|\mathbb{E}\!\left[ \widetilde{\bm{x}}_k-\bm{x}_{\mathrm{LS}}\right]\right\|_2$ where the four methods RK, REK, RBK, and RDBK are applied to the doubly noisy linear system.
Section~\ref{sec6} studies both
$\left\|\mathbb{E}\!\left[\widetilde{\bm{x}}_k-\bm{x}_{\mathrm{LS}}\right]\right\|_B$ and $\mathbb{E}\!\left[\left\|\widetilde{\bm{x}}_k-\bm{x}_{\mathrm{LS}}\right\|_B^2\right]$ under the more general sketch-and-project method for solving the doubly noisy linear system. Finally, Section~\ref{sec5} presents numerical experiments that support our theoretical results.

\section{Summary of Randomized Kaczmarz for Noisy Linear Systems}\label{sec2}
In this section, we briefly review existing studies on the randomized Kaczmarz method for noisy linear systems. The following theorem concerns a single noisy linear system \eqref{1.3} with the right-hand side corrupted by noise $\bm \epsilon$.

\begin{theorem}\cite[Theorem 2.1]{Need-2010}\label{thm2.1}
Assume that $A$ has full column rank and the noiseless system $A \bm x=\bm b$ is consistent with $\bm {x}_{LS}=A^{\dagger}\bm b$. Let $\bm {\tilde x}_{k}$ be the k-th iterate of the randomized Kaczmarz method applied to the single noisy linear system $A \bm x = \bm{\tilde{b}}$ with $\bm{\tilde{b}}=\bm b+\bm \epsilon$, where $\bm \epsilon=(\epsilon_1,\epsilon_2,\ldots,\epsilon_m)^\top$ and $A=(\bm a_1^\top,\bm a_2^\top,\ldots,\bm a_m^\top)^\top$. Then we have
$$
\mathbb{E}\|\bm {\tilde x}_{k}-\bm {x}_{LS}\|_2^2\leq\left(1-\frac{1}{\mu}\right)^{k} \|\bm {\tilde x}_0-\bm {x}_{LS}\|_2^2+\mu\gamma^2,
$$
where $\mu=\|A^{\dagger}\|_2^2\|A\|_F^2$ and $\gamma= \max\limits_{i \in \{1,2,\ldots,m\}}\left\{\frac{|\epsilon_i|}{\|\bm {a_i}\|_2}\right\}$.
\end{theorem}

\begin{remark}
{\it  Theorem \ref{thm2.1} shows that the convergence rate $1-1/\mu$ of the RK method applied to the single noisy linear system \eqref{1.3} is the same as that of the RK method applied directly to the system \eqref{1.2}, while the convergence horizon $\mu\gamma^2$ depends on the noise level and the condition of the coefficient matrix $A$.}
\end{remark}

In \cite{Zouzias-2013}, Zouzias and Freris further remove the full-rank assumption and improve the results of Theorem \ref{thm2.1}.	
\begin{theorem}\cite[Theorem 3.7]{Zouzias-2013}\label{thm2.2}
Let $\bm {\tilde x}_{k}$ be the $k$-th iterate of the RK method applied to the single noisy linear system $A \bm x = \bm{\tilde{b}}$ with $\bm{\tilde{b}}=\bm b+\bm \epsilon$ for any fixed $\bm \epsilon \in \mathbb{R}^{m}$. Assume that the system $A \bm x=\bm b$ has a solution $\bm {x}_{LS}=A^{\dagger}\bm b$ for some $\bm b \in \mathbb{R}^{m}$. If $\bm {\tilde x}_0\in \mathcal{R}(A^{\top})$, then we have
	$$
\mathbb{E}\|\bm {\tilde x}_k-\bm {x}_{LS}\|_2^2\leq\left(1-\frac{1}{\mu}\right)^{k} \|\bm {\tilde x}_0-\bm {x}_{LS}\|_2^2+\frac{\|\bm \epsilon\|_2^2}{\sigma_{\min}^2(A)},
	$$
	where $\mu=\|A^{\dagger}\|_2^2\|A\|_F^2$.
\end{theorem}

In real-world applications, not only the right-hand side $\bm b$ but the coefficient matrix $A$ are corrupted by noise. The following theorem shows the convergence of the RK method applied to the doubly noisy linear system \eqref{1.1}. Unlike Theorems \ref{thm2.1} and \ref{thm2.2}, it indicates that the convergence rate of the RK method applied to the doubly noisy linear system \eqref{1.1} is independent of the noiseless coefficient matrix, and depends only on a scaled condition number of the noisy coefficient matrix $\widetilde{A}$, while the convergence horizon depends on the perturbations $E$ and $\bm \epsilon$.
\begin{theorem}\cite[Theorem 3.1]{Bergou-2024}\label{thm2.3}
Let $\bm {{\tilde x}}_{k}$ be the $k$-th iterate of the RK method applied to the doubly noisy linear system $\widetilde{A} \bm {x} = \bm{\tilde{b}}$, with $\bm {\tilde x}_0- \bm {x}_{LS} \in \mathcal{R}(\widetilde{A}^{\top})$. Assume that the noiseless system $A \bm x=\bm b$ is consistent with $\bm {x}_{LS}=A^{\dagger}\bm b$. Then we have
	$$
	\mathbb{E}\|\bm {{\tilde x}}_{k}-\bm {x}_{LS}\|_2^2\leq\left(1-\frac{1}{\widetilde{\mu}}\right)^{k} \|\bm {\tilde x}_0-\bm {x}_{LS}\|_2^2+\frac{\|E\bm {x}_{LS}-\bm \epsilon\|_2^2}{\sigma^2_{\min}(\widetilde A)},
	$$
	where $\widetilde \mu=\|\widetilde{A}^{\dagger}\|_2^2\|\widetilde{A}\|_F^2$, $E=\widetilde{A}-A$, and $\bm \epsilon=\bm{\tilde{b}}-\bm b$.
\end{theorem}

Without imposing any initial assumptions, Bergou {\it et al.} proved that RK applied to \eqref{1.1} converges to a ball centered at \scalebox{0.9}{$(\bm {\tilde x}_0)_{\mathcal{N}(\widetilde{A})}+(\bm{x}_*)_{\mathcal{R}(\widetilde{A}^\top)}$}, where $\bm {x}_* \in \mathbb{R}^{n}$ is arbitrary.

\begin{theorem}\cite[Theorem 5]{Bergou-2025}\label{thm2.4}
Let $\bm {\tilde x}_{k}$ be the $k$-th iterate of the RK algorithm applied to the doubly noisy linear system $\widetilde{A} \bm {x} = \bm{\tilde{b}}$, where $\widetilde{A}$ and $\bm{\tilde{b}}$ are fixed. Let $\bm {\tilde x}_0$ be arbitrary. Then we have, for any $\bm {x}_* \in \mathbb{R}^{n}$,
\[
\thickmuskip=1mu \medmuskip=0.8mu 
\begin{aligned}
	\mathbb{E}\left\|\bm {\tilde x}_{k}-(\bm {\tilde x}_0)_{\mathcal{N}(\widetilde{A})}-(\bm{x}_*)_{\mathcal {R}(\widetilde{A}^\top)}\right\|_2^2\leq\left(1-\frac{1}{\widetilde{\mu}}\right)^{k} \left\|(\bm{\tilde x}_0)_{\mathcal{R}(\widetilde{A}^\top)}-(\bm{x}_*)_{\mathcal {R}(\widetilde{A}^\top)}\right\|_2^2+\frac{\|\widetilde{A} \bm {x}_*-\bm{\tilde{b}}\|_2^2}{{\sigma}^2_{\min}(\widetilde A)},
\end{aligned}
\]
where $\widetilde{\mu}=\|\widetilde{A}^{\dagger}\|_2^2\|\widetilde{A}\|_F^2$.	
\end{theorem}

Moreover, the convergence of the mean of the RK iterates is investigated in \cite{Bergou-2025}. As the following theorem indicates, the bound yields a faster convergence rate than that given in Theorem \ref{thm2.4}.
\begin{theorem}\cite[Corollary 2]{Bergou-2025}\label{thm2.5}
Let $\bm {\tilde x}_{k}$ be the $k$-th iterate of the RK method applied to the doubly noisy linear system $\widetilde{A} \bm {x} = \bm{\tilde{b}}$. Let $\bm {\tilde x}_0$ be arbitrary. Then we have, for any $\bm {x}_* \in \mathbb{R}^{n}$,
\[
\thickmuskip=1mu \medmuskip=0.8mu 
\begin{aligned}
	\left\|\mathbb{E}[\bm {\tilde x}_k]-(\bm {\tilde x}_0)_{\mathcal {N}(\widetilde{A})}-(\bm{x}_*)_{\mathcal {R}(\widetilde{A}^\top)}\right\|_2\leq\left(1-\frac{1}{\widetilde{\mu}}\right)^{k} \left\|(\bm{\tilde x}_0)_{\mathcal{R}(\widetilde{A}^\top)}-(\bm{x}_*)_{\mathcal {R}(\widetilde{A}^\top)}\right\|_2+\frac{\|\widetilde{A} \bm {x}_*-\bm{\tilde{b}}\|_2}{{\sigma}_{\min}(\widetilde A)},
\end{aligned}
\]
	where $\widetilde{\mu}=\|\widetilde{A}^{\dagger}\|_2^2\|\widetilde{A}\|_F^2$.	
\end{theorem}

\begin{remark}
{\it Two remarks are in order. On the one hand, it is required that $\bm {\tilde x}_0- \bm {x}_{LS} \in \mathcal{R}(\widetilde{A}^{\top})$ in Theorem \ref{thm2.3}, however, $\bm {x}_{LS}$ is unknown {\it a prior}. Moreover, the assumption will not be satisfied in general if $\widetilde{A}$ is a flat matrix with $m<n$.
On the other hand, although Theorem \ref{thm2.4} and Theorem \ref{thm2.5} do not make any assumptions on the initial guess, their bounds are NOT for the distance $\|\bm {{\tilde x}}_{k}-\bm {x}_{LS}\|_2$ between the approximation $\bm {\tilde x}_{k}$ for the doubly noisy linear system \eqref{1.1} and the LS solution $\bm {x}_{LS}$ of the noiseless linear system \eqref{1.2}, which is the quantity of interest in real-world applications. Thus, it is necessary to establish new results of the convergence behavior of randomized Kaczmarz-type methods for the doubly noisy linear system \eqref{1.1}.
}
\end{remark}

\section{A revisit to the expectation of the error norm for the RK solving the doubly noisy linear system}\label{sec4}
In this section, we revisit the convergence behavior of the RK method applied to the doubly noisy linear system \eqref{1.1}. Note that the following theorem does not require that the noiseless linear system \eqref{1.2} is consistent. 



\begin{theorem}\label{thm3.5}
	Let $\bm {\tilde x}_{k}$ be the $k$-th iterate of the RK method applied to the doubly noisy linear system \eqref{1.1}. Let $\bm {x}_{LS}=A^{\dagger}\bm b$ and $\bm {\tilde x}_{LS}=\widetilde{A}^{\dagger}\bm {\tilde b}$ denote the smallest norm least squares solution of the noiseless linear system \eqref{1.2} and the doubly noisy linear system \eqref{1.1}, respectively. Let \(\bm{e}_{k}=\bm{\tilde x}_{k}-\bm{x}_{LS}\), then 
	
	{\tt (i)} If $\bm{e}_{k-1} \in \mathcal{R}(\widetilde{A}^{\top})$, then
	$$\mathbb{E}\|\bm{e}_{k}\|_2^2 \leq \left(1-\frac{{\sigma}^2_{\min}(\widetilde A)}{\|\widetilde{A}\|_F^2} \right)\mathbb{E} \| \bm{e}_{k-1}\|_2^2+\frac{\|\widetilde{A}\bm{x}_{LS} - \bm {\tilde b} \|_2^2}{\|\widetilde{A}\|^2_F}.$$
	
	{\tt (ii)} 	If  $\bm{e}_{k-1} \in \mathcal{N}(\widetilde{A})$,  then
	$$
    \mathbb{E}\|\bm{e}_{k}\|_2^2=\mathbb{E}\|\bm{e}_{k-1}\|_2^2+\frac{\|\widetilde{A}\bm{x}_{LS} - \bm {\tilde b} \|_2^2}{\|\widetilde{A}\|^2_F} 
    $$
	when the index set $i_k$ is randomly selected such that $\bm{\tilde{a}}^\top_{i_k} \bm{x}_{LS} - \tilde{b}_{i_k}\neq 0$,  and $\bm{\tilde x}_{k}=\bm{\tilde x}_{k-1}$ if $\bm{\tilde{a}}^\top_{i_k} \bm{x}_{LS} - \tilde{b}_{i_k}=0$.
	
	{\tt (iii)} If $\bm{e}_{k-1} \notin \mathcal{R}(\widetilde{A}^{\top})$ and $\bm{e}_{k-1} \notin \mathcal{N}(\widetilde{A})$, then
	
	$$\mathbb{E}\|\bm{e_{k}}\|_2^2\leq \left(1-\frac{{\sigma}^2_{\min}(\widetilde A)}{\|\widetilde{A}\|_F^2}\cdot \xi_{k-1}\right)\mathbb{E}\| \bm {e_{k-1}}\|_2^2+\frac{\|\widetilde{A}\bm{x}_{LS} - \bm {\tilde b} \|_2^2}{\|\widetilde{A}\|^2_F},$$
	where $\xi_{k-1}=\frac{\mathbb{E}\left\|(\bm{\tilde x}_{k-1} - \bm{x}_{LS})_{\mathcal{R}(\widetilde{A}^\top)}\right\|_2^2}{\mathbb{E}\left\|(\bm{\tilde x}_{k-1} - \bm{x}_{LS})\right\|_2^2}<1$.

\end{theorem}

\begin{proof}
	Applying the RK method to \eqref{1.1}, we get
	\[
	\begin{aligned}
		\bm{\tilde x}_k - \bm{x}_{LS}
		&= \left( I - \frac{\bm{\tilde{a}}_{i_k} \bm{\tilde a}^\top_{i_k}}{\|\bm{\tilde a}_{i_k}\|_2^2} \right) (\bm{\tilde x}_{k-1} - \bm{x}_{LS}) -\frac{(\bm{\tilde a}^\top_{i_k} \bm{x}_{LS} - \tilde{b}_{i_k})}{\|\bm{\tilde a}_{i_k}\|_2^2} \cdot \bm{\tilde a}_{i_k}.
	\end{aligned}
	\]
	Let \(\bm{e}_{k} = \bm{\tilde x}_k - \bm{x}_{LS}\), then 
	\begin{align}\label{eqthm3.5.1}
		\|\bm{e}_{k}\|_2^2&= \left\|\left( I - \frac{\bm{\tilde{a}}_{i_k} \bm{\tilde a}^\top_{i_k}}{\|\bm{\tilde a}_{i_k}\|_2^2} \right) \bm{e}_{k-1} - \frac{(\bm{\tilde a}^\top_{i_k} \bm{x}_{LS} - \tilde{b}_{i_k})}{\|\bm{\tilde a}_{i_k}\|_2^2} \cdot \bm{\tilde a}_{i_k} \right\|_2^2\nonumber\\
		&=\left\|\left( I - \frac{\bm{\tilde a}_{i_k} \bm{\tilde a}^\top_{i_k}}{\|\bm{\tilde a}_{i_k}\|_2^2} \right) \bm{e}_{k-1} \right\|_2^2+\left\| \frac{(\bm{\tilde a}^\top_{i_k} \bm{x}_{LS} - \tilde{b}_{i_k})}{\|\bm{\tilde a}_{i_k}\|_2^2} \cdot \bm{\tilde a}_{i_k} \right\|_2^2\nonumber\\
		&=\bm{e}^\top_{k-1}\left( I - \frac{\bm{\tilde{a}}_{i_k} \bm{\tilde a}^\top_{i_k}}{\|\bm{\tilde a}_{i_k}\|_2^2} \right) \bm{e}_{k-1} + \frac{(\bm{\tilde a}^\top_{i_k} \bm{x}_{LS} - \tilde{b}_{i_k})^2}{\|\bm{\tilde a}_{i_k}\|_2^2}.	
	\end{align}
	Taking conditional expectation on both sides of the above equation, we obtain
	\begin{align*}
		\mathbb{E}_{k-1}\|\bm{e}_{k}\|_2^2&={\small\bm{e}^\top_{k-1}\left( I -\sum_{i_{k}=1}^{m}\frac{\left\| \bm{\tilde a}_{i_k}\right\|_2^2}{\| \tilde{A} \|_F^2}\cdot \frac{\bm{\tilde{a}}_{i_k} \bm{\tilde a}^\top_{i_k}}{\|\bm{\tilde a}_{i_k}\|_2^2} \right) \bm{e}_{k-1} +\sum_{i_{k}=1}^{m}\frac{\left\| \bm{\tilde a}_{i_k}\right\|_2^2}{\| \tilde{A} \|_F^2}\cdot \frac{(\bm{\tilde a}^\top_{i_k} \bm{x}_{LS} - \tilde{b}_{i_k})^2}{\|\bm{\tilde a}_{i_k}\|_2^2}}\\
		&=\bm{e}^\top_{k-1} \left(I-\frac{\widetilde{A}^\top\widetilde{A}}{\|\widetilde{A}\|^2_F}\right) \bm {e}_{k-1}+\frac{\|\widetilde{A}\bm{x}_{LS} - \bm {\tilde b} \|_2^2}{\|\widetilde{A}\|^2_F}.
	\end{align*}
	
{\tt (i)} If $\bm{e}_{k-1}=\bm{\tilde x}_{k-1} - \bm{x}_{LS} \in \mathcal{R}(\widetilde{A}^\top)$, then we have $ \|\tilde{A} \bm e_{k-1}\|_2 \geq {\sigma}_{\min}(\widetilde A)\cdot\|\bm e_{k-1}\|_2$, and
	\begin{align*}
		\mathbb{E}\|\bm{e}_{k}\|_2^2&=\mathbb{E}\left[	\mathbb{E}_{k-1}\|\bm{e}_{k}\|_2^2\right]=\mathbb{E}\left[\|\bm {e}_{k-1}\|_2^2-\frac{\|\tilde{A} \bm e_{k-1}\|_2^2}{\|\widetilde{A}\|_F^2}+\frac{\|\widetilde{A}\bm{x}_{LS} - \bm {\tilde b} \|_2^2}{\|\widetilde{A}\|^2_F}\right]\\
		&\leq \left(1-\frac{{\sigma}^2_{\min}(\widetilde A)}{\|\widetilde{A}\|_F^2} \right)\mathbb{E}\|\bm {e}_{k-1}\|_2^2+\frac{\|\widetilde{A}\bm{x}_{LS} - \bm {\tilde b} \|_2^2}{\|\widetilde{A}\|^2_F}.
	\end{align*}

{\tt (ii)} If  $\bm{e}_{k-1}=\bm{\tilde x}_{k-1} - \bm{x}_{LS} \in \mathcal{N}(\widetilde{A})$, then we have $ \|\tilde{A} \bm {e}_{k-1}\|_2 = 0$, i.e., $\bm{\tilde{a}}_{i_k}^T \bm {e}_{k-1}=0$, $i_k \in \{1,2,\ldots,m\}$. It follows from \eqref{eqthm3.5.1} that 
$$
\|\bm{e}_{k}\|_2^2=\bm{e}^\top_{k-1}\left( I - \frac{\bm{\tilde{a}}_{i_k} \bm{\tilde a}^\top_{i_k}}{\|\bm{\tilde a}_{i_k}\|_2^2} \right) \bm{e}_{k-1} + \frac{(\bm{\tilde a}^\top_{i_k} \bm{x}_{LS} - \tilde{b}_{i_k})^2}{\|\bm{\tilde a}_{i_k}\|_2^2}=\|\bm{e}_{k-1}\|_2^2+\frac{(\bm{\tilde a}^\top_{i_k} \bm{x}_{LS} - \tilde{b}_{i_k})^2}{\|\bm{\tilde a}_{i_k}\|_2^2}.
$$
	
	On the one hand, if $\bm{\tilde a}^\top_{i_k} \bm{x}_{LS} - \tilde{b}_{i_k}=0$ at the $k$-th iterate, then $\bm{\tilde x}_{k}=\bm{\tilde x}_{k-1}$. 	
	On the other hand, if $\bm{\tilde{a}}^\top_{i_k} \bm{x}_{LS} - \tilde{b}_{i_k}\neq 0$, then 
	\begin{align*}
	\mathbb{E}\|\bm{e}_{k}\|_2^2&=\mathbb{E}\|\bm{e}_{k-1}\|_2^2+\sum_{i_{k}=1}^{m}\frac{\left\| \bm{\tilde a}_{i_k}\right\|_2^2}{\| \tilde{A} \|_F^2}\cdot
	\frac{(\bm{\tilde a}^\top_{i_k} \bm{x}_{LS} - \tilde{b}_{i_k})^2}{\|\bm{\tilde a}_{i_k}\|_2^2}\\
	&=\mathbb{E}\|\bm{e}_{k-1}\|_2^2+\frac{\|\widetilde{A}\bm{x}_{LS} - \bm {\tilde b} \|_2^2}{\|\widetilde{A}\|^2_F}.
\end{align*}

{\tt (iii)} In this case, we notice that $\bm{e}_{k-1}=\bm{\tilde x}_{k-1} - \bm{x}_{LS}=(\bm{\tilde x}_{k-1} -\bm{ x}_{LS})_{\mathcal {R}(\widetilde{A}^\top)}+(\bm{\tilde x}_{k-1} - \bm {x}_{LS})_{\mathcal {N}(\widetilde{A})}$. Thus, $$ \|\tilde{A} \bm {e}_{k-1}\|_2 = \left\|\tilde{A}(\bm{\tilde x}_{k-1} - \bm{x}_{LS})_{\mathcal {R}(\widetilde{A}^\top)}\right\|_2 \geq {\sigma}_{\min}(\widetilde A)\cdot\left\|(\bm{\tilde x}_{k-1} - \bm{x}_{LS})_{\mathcal {R}(\widetilde{A}^\top)}\right\|_2.
$$ 
Let $$\xi_{k-1}=\frac{\mathbb{E}\left\|(\bm{\tilde x}_{k-1} - \bm{x}_{LS})_{\mathcal{R}(\widetilde{A}^\top)}\right\|_2^2}{\mathbb{E}\left\|(\bm{\tilde x}_{k-1} - \bm{x}_{LS})\right\|_2^2},$$
then we have that
\begin{align*}
	&\mathbb{E}\|\bm{e}_{k}\|_2^2=\mathbb{E}\left[	\mathbb{E}_{k-1}\|\bm{e}_{k}\|_2^2\right]=\mathbb{E}\left[\|\bm {e}_{k-1}\|_2^2-\frac{\|\tilde{A} \bm e_{k-1}\|_2^2}{\|\widetilde{A}\|_F^2}+\frac{\|\widetilde{A}\bm{x}_{LS} - \bm {\tilde b} \|_2^2}{\|\widetilde{A}\|^2_F}\right]\\
	&\leq \mathbb{E} \left[\|\bm {e}_{k-1}\|_2^2\right]- \frac{{\sigma}^2_{\min}(\widetilde A)}{\|\widetilde{A}\|_F^2}\cdot\mathbb{E} \left[\left\|(\bm{\tilde x}_{k-1} - \bm{x}_{LS})_{\mathcal {R}(\widetilde{A}^\top)}\right\|_2^2\right]+\frac{\|\widetilde{A}\bm{x}_{LS}-\bm {\tilde b} \|_2^2}{\|\widetilde{A}\|^2_F}\\
	&=\left(1-\frac{{\sigma}^2_{\min}(\widetilde A)}{\|\widetilde{A}\|_F^2}\cdot \xi_{k-1}\right)\mathbb{E}\|\bm {e}_{k-1}\|_2^2+\frac{\|\widetilde{A}\bm{x}_{LS} - \bm {\tilde b} \|_2^2}{\|\widetilde{A}\|^2_F}.
\end{align*}

\end{proof}

\begin{remark}\label{re3.3}
{\it Three remarks are in order. First, Theorem \ref{thm3.5} removes the strict condition \scalebox{0.95}{$\bm {\tilde x}_0- \bm {x}_{LS} \in \mathcal{R}(\widetilde{A}^\top)$} that is required in Theorem \ref{thm2.3}. Compared with Theorem \ref{thm2.4}, Theorem \ref{thm3.5} establishes the upper bound of \scalebox{0.95}{$\mathbb{E}\|\bm {{\tilde x}}_{k}-\bm {x}_{LS}\|_2$}, which we are interested in.
Second, in cases {\tt(i)} and {\tt(iii)}, the RK method converges to a ball centered at $\bm{x}_{LS}$. As the convergence factor $0<\xi_{k-1}<1$ in case {\tt(iii)}, its convergence rate may be slower than that in case {\tt(i)}. Third, if case {\tt(ii)} occurs, the RK method will not converge. In fact, if $\widetilde{A}$ has full column rank, case {\tt (ii)} occurs if and only if $\bm e_{k-1}=0$, implying $\bm{\tilde{x}}_{k-1}=\bm{x}_{LS}$. 
}	
\end{remark}

\section{Main results in terms of the behavior of the norm of the expected error}\label{sec3}


Considering {\tt (iii)} of Theorem \ref{thm3.5}, which is the general case in practice, if $\xi_{k-1}$ approaches to zero, then the result implies the RK method may have a much slower convergence rate than the real situation.
From Jensen's inequality, we have that 
\begin{equation}\label{jensen}
\left\|\mathbb{E}[\bm {\tilde x} _{k}-\bm{x}_{LS}]\right\|_2 \leq \sqrt{\mathbb{E}[\|\bm {\tilde x}_{k}-\bm{x}_{LS}\|_2^2]}.
\end{equation}
Thus, the norm of the expectation and the expectation of the norm are closely related. In this section, we focus on the behavior of $\left\|\mathbb{E}[\bm {\tilde x} _{k}-\bm{x}_{LS}]\right\|$ for four methods as $k$ increases. 

\subsection{On convergence of RK and REK for the doubly noisy linear system}\label{sec3.1}
In this subsection, we consider the behavior of $\|\mathbb{E}[\bm {\tilde x}_{k}-\bm{x}_{LS}]\|_2$ of the RK method and the REK method for the doubly noisy linear system. RK and REK are two popular Kaczmarz methods for consistent and inconsistent linear systems, respectively.
We first need the following lemma for subsequent analysis. The proof is straightforward and is omitted.
\begin{lemma}\label{lem3.1}
Let matrix $\widetilde{A} \in \mathbb{R}^{m\times n}$ with $\mathrm{rank}(\widetilde{A})=r$.
On the one hand, if $\bm {x} \in \mathcal{R}(\widetilde{A}^{\top})$, then we have
\begin{equation}\label{a1}
\left\|\left[I-\left(I-\frac{\widetilde{A}^{\top}\widetilde{A}}{\|\widetilde{A}\|^2_F}\right)^k\right]\bm {x}\right\|_2 \leq \left[1-\left(1-\frac{{\sigma}^2_{\max}(\widetilde A)}{\|\widetilde{A}\|^2_F}\right)^k\right]\|\bm {x}\|_2,
\end{equation}
and
\begin{equation}\label{a2}
 \left\|\left(I-\frac{\widetilde{A}^{\top}\widetilde{A}}{\|\widetilde{A}\|^2_F}\right)^k \bm {x}\right\|_2 \leq \left(1-\frac{{\sigma}^2_{\min}(\widetilde A)}{\|\widetilde{A}\|^2_F}\right)^k\|\bm {x}\|_2.
\end{equation}

On the other hand, if $\bm {x} \in \mathcal{N}(\widetilde{A})$, then we have
\begin{equation}\label{a4}
\left\|\left[I-\left(I-\frac{\widetilde{A}^{\top}\widetilde{A}}{\|\widetilde{A}\|^2_F}\right)^k\right]\bm {x}\right\|_2 =0.
\end{equation}
and
\begin{equation}\label{a3}
\left\|\left(I-\frac{\widetilde{A}^{\top}\widetilde{A}}{\|\widetilde{A}\|^2_F}\right)^k \bm {x}\right\|_2 =\|\bm {x}\|_2,
\end{equation}

\end{lemma}

We are ready to show the behavior of $\|\mathbb{E}[\bm {\tilde x}_{k}-\bm{x}_{LS}]\|_2$ obtained from the RK method, for solving the doubly noisy linear system \eqref{1.1}. Different from \cite[Theorem 2.3]{Bergou-2024}, we remove the initial assumption and the consistency of the original noiseless system is unnecessary. 
Moreover, unlike \cite[Theorem 2.5]{Bergou-2025}, we make use of the smallest norm least squares solution $\bm{x}_{LS}$ as a limiting reference point. This is because we are interested in how far the Kaczmarz iterate $\bm {\tilde x}_{k}$ applied to \eqref{1.1} is from the true solution in practical applications. 
\begin{theorem}\label{thm3.1}
Let $\bm {\tilde x}_{k}$ be the $k$-th iterate of the RK method applied to the doubly noisy linear system \eqref{1.1}. Let $\bm {x}_{LS}=A^{\dagger}\bm b$ and $\bm {\tilde x}_{LS}=\widetilde{A}^{\dagger}\bm {\tilde b}$ be the smallest norm least squares solution of the noiseless linear system \eqref{1.2} and the doubly noisy linear system \eqref{1.1}, respectively. For any initial point $\bm {\tilde x}_{0}\in \mathbb{R}^{n}$ and iteration number $k \in \mathbb{N}^+$, we have
\begin{align*}
\|\mathbb{E} (\bm {\tilde x}_{k}-\bm {x}_{LS})\|_2\leq \alpha^{k}\left\|(\bm {\tilde x}_0-\bm {x}_{LS})_{\mathcal {R}(\widetilde{A}^\top)}\right\|_2&+\beta \big\|(\bm {x}_{LS}-\bm {\tilde x}_{LS})_{\mathcal {R}(\widetilde{A}^\top)}\big\|_2\\
&+\big\|(\bm {\tilde x}_0-\bm {x}_{LS})_{\mathcal {N}(\widetilde{A})}\big\|_2,\nonumber
\end{align*}
where $\alpha=1-\frac{{\sigma}^2_{\min}(\widetilde A)}{\|\widetilde{A}\|^2_F}$ and $\beta=1-\left(1-\frac{{\sigma}^2_{\max}(\widetilde A)}{\|\widetilde{A}\|^2_F}\right)^k$.
\end{theorem}
\begin{proof}
Applying the RK method to \eqref{1.1}, it follows that
	\begin{align*}
	\bm{\tilde x}_k - \bm{x}_{LS} &= \bm{\tilde x}_{k-1} - \bm{x}_{LS} - \frac{\bm{\tilde{a}}^\top_{i_k} (\bm{\tilde x}_{k-1} - \bm{x}_{LS}) + \bm{\tilde{a}}^\top_{i_k} \bm{x}_{LS} - \tilde{b}_{i_k}}{\|\bm{\tilde{a}}_{i_k}\|_2^2} \cdot \bm{\tilde{a}}_{i_k} \\
	&= \left( I - \frac{\bm{\tilde{a}}_{i_k} \bm{\tilde{a}}^\top_{i_k}}{\|\bm{\tilde{a}}_{i_k}\|_2^2} \right) (\bm{\tilde x}_{k-1} - \bm{x}_{LS}) - \frac{\bm{\tilde{a}}^\top_{i_k} \bm{x}_{LS} - \tilde{b}_{i_k}}{\|\bm{\tilde{a}}_{i_k}\|_2^2} \cdot \bm{\tilde{a}}_{i_k}.
\end{align*}
Let \(\bm{e}_{k} = \bm{\tilde x}_{k} - \bm{x}_{LS}\), then
{\small	
\[
\begin{aligned}
\mathbb{E}_{k-1}\left[\bm{e}_{k}\right]&= \left( I -\sum_{i_{k}=1}^{m}\frac{\left\| \bm{\tilde a}_{i_k}\right\|_2^2}{\| \tilde{A} \|_F^2}\cdot \frac{\bm{\tilde{a}}_{i_k} \bm{\tilde{a}}^\top_{i_k}}{\|\bm{\tilde{a}}_{i_k}\|_2^2} \right) \bm{e}_{k-1} -\sum_{i_{k}=1}^{m}\frac{\left\| \bm{\tilde{a}}_{i_k}\right\|_2^2}{\| \tilde{A} \|_F^2}\cdot \frac{\bm{\tilde{a}}^\top_{i_k} \bm{x}_{LS} - \tilde{b}_{i_k}}{\|\bm{\tilde{a}}_{i_k}\|_2^2} \cdot \bm{\tilde{a}}_{i_k}\\
&=\left(I-\frac{\widetilde{A}^\top\widetilde{A}}{\|\widetilde{A}\|^2_F}\right) \bm {e}_{k-1}-\frac{\widetilde{A}^\top(\widetilde{A}\bm{x}_{LS} - \bm{\tilde b} )}{\|\widetilde{A}\|^2_F}\\
&=\left(I-\frac{\widetilde{A}^\top\widetilde{A}}{\|\widetilde{A}\|^2_F}\right) \bm {e}_{k-1}-\frac{\widetilde{A}^\top\widetilde{A}(\bm{x}_{LS} - \bm {\tilde x}_{LS} )}{\|\widetilde{A}\|^2_F},	
\end{aligned}
\]
}
where the last equation uses the normal equation $\widetilde{A}^\top\widetilde{A}\bm {\tilde x}_{LS}=\widetilde{A}^\top\bm {\tilde b} $. By the law of total expectation, we have
{\small
\[
\begin{aligned}
\mathbb{E}\left[\bm{e}_{k}\right]&=\mathbb{E}\left[\mathbb{E}_{k-1}^i\left[\bm{e}_{k}\right]\right]=\left(I-\frac{\widetilde{A}^\top\widetilde{A}}{\|\widetilde{A}\|^2_F}\right) \mathbb{E}\left[\bm {e}_{k-1}\right]-\frac{\widetilde{A}^\top\widetilde{A}(\bm{x}_{LS} - \bm {\tilde x}_{LS})}{\|\widetilde{A}\|^2_F}\\
	&={\small\left(I-\frac{\widetilde{A}^\top\widetilde{A}}{\|\widetilde{A}\|^2_F}\right)^2 \mathbb{E}\left[\bm {e}_{k-2}\right]-\left(I-\frac{\widetilde{A}^\top\widetilde{A}}{\|\widetilde{A}\|^2_F}\right) \frac{\widetilde{A}^\top\widetilde{A}(\bm{x}_{LS} - \bm {\tilde x}_{LS} )}{\|\widetilde{A}\|^2_F}-\frac{\widetilde{A}^\top\widetilde{A}(\bm{x}_{LS} - \bm {\tilde x}_{LS} )}{\|\widetilde{A}\|^2_F}}\\
	&=\left(I-\frac{\widetilde{A}^\top\widetilde{A}}{\|\widetilde{A}\|^2_F}\right)^k \bm {e}_{0}-\sum_{j=1}^{k-1}\left(I-\frac{\widetilde{A}^\top\widetilde{A}}{\|\widetilde{A}\|^2_F}\right)^j \frac{\widetilde{A}^\top\widetilde{A}(\bm{x}_{LS} - \bm {\tilde x}_{LS} )}{\|\widetilde{A}\|^2_F}\\
	&=\left(I-\frac{\widetilde{A}^\top\widetilde{A}}{\|\widetilde{A}\|^2_F}\right)^k \bm {e}_{0}-\left[I-\left(I-\frac{\widetilde{A}^\top\widetilde{A}}{\|\widetilde{A}\|^2_F}\right)^k\right] (\bm{x}_{LS} - \bm {\tilde x}_{LS}).\\
\end{aligned}
\]
}

Denote by $T=I-\frac{\widetilde{A}^\top\widetilde{A}}{\|\widetilde{A}\|^2_F}$, $\alpha=1-\frac{{\sigma}^2_{\min}(\widetilde A)}{\|\widetilde{A}\|^2_F}$, and $\beta=1-\left(1-\frac{{\sigma}^2_{\max}(\widetilde A)}{\|\widetilde{A}\|^2_F}\right)^k$, we have from Lemma \ref{lem3.1} that
{\small
	\[
	\begin{aligned}
\left\|\mathbb{E}\left[\bm{e}_{k}\right]\right\|_2&=\left\|T^k \cdot \bm {e}_{0}-\left(I-T^k\right)\cdot (\bm{x}_{LS} - \bm {\tilde x}_{LS})\right\|_2\\
& \leq \left\| \left(I-T^k\right)\cdot(\bm{x}_{LS} - \bm {\tilde x}_{LS}) \right\|_2+\left\|T^k \cdot \bm {e}_{0}\right\|_2\\
& \leq \left\| \left(I-T^k\right)\cdot (\bm{x}_{LS} - \bm {\tilde x}_{LS})_{\mathcal {R}(\widetilde{A}^\top)}\right\|_2+\left\|\left(I-T^k\right)\cdot(\bm{x}_{LS} - \bm {\tilde x}_{LS})_{\mathcal {N}(\widetilde{A})}\right\|_2\\
&\quad +\left\|T^k\cdot\left(\bm{e}_0\right)_{\mathcal {R}(\widetilde{A}^\top)}\right\|_2+\left\|T^k\cdot\left(\bm {e}_0\right)_{\mathcal {N}(\widetilde{A})}\right\|_2\\
&\leq (1-\beta^k)\left\|(\bm{x}_{LS} - \bm {\tilde x}_{LS})_{\mathcal {R}(\widetilde{A}^\top)} \right\|_2+\alpha^{k} \left\|\left(\bm{e}_0\right)_{\mathcal {R}(\widetilde{A}^\top)}\right\|_2+ \left\|\left(\bm {e}_0\right)_{\mathcal {N}(\widetilde{A})}\right\|_2\\
&=\alpha^{k}\left\|(\bm {\tilde x}_0-\bm {x}_{LS})_{\mathcal {R}(\widetilde{A}^\top)}\right\|_2+\beta \left\|(\bm{x}_{LS} - \bm {\tilde x}_{LS})_{\mathcal {R}(\widetilde{A}^\top)} \right\|_2+ \left\|\left(\bm {\tilde x}_0-\bm {x}_{LS}\right)_{\mathcal {N}(\widetilde{A})}\right\|_2.
\end{aligned}
\]}
\end{proof}

\begin{remark}\label{rem4.1}
{\it  A direct consequence of Theorem \ref{thm3.1} is that there is no need to worry about the influence of the factor $\xi_{k-1}$ appearing in Theorem \ref{thm3.5} for the convergence. 
More precisely, Theorem \ref{thm3.1} indicates that the RK method for \eqref{1.1} converges to a ball centered at the smallest norm least squares solution $\bm{x}_{LS}$ of the original noiseless system. The radius of this ball is referred to as the convergence horizon, which is determined by
\scalebox{0.85}{$\big\|(\bm{x}_{LS}-\bm{\tilde{x}}_{LS})_{\mathcal{R}(\widetilde{A}^{\top})}\big\|_2+\big\|(\bm{\tilde{x}}_0-\bm{x}_{LS})_{\mathcal{N}(\widetilde{A})}\big\|_2$}.
During the initial stage of the iteration, the convergence rate of RK is influenced not only by the quantity
\scalebox{0.9}{$\sigma_{\min}^2(\widetilde{A})/\|\widetilde{A}\|_F^2$},
but also by
\scalebox{0.9}{$\sigma_{\max}^2(\widetilde{A})/\|\widetilde{A}\|_F^2$}.
However, as the iteration number $k$ becomes sufficiently large, the convergence rate is dominated by the quantity
\scalebox{0.9}{$\sigma_{\min}^2(\widetilde{A})/\|\widetilde{A}\|_F^2$}.
}
\end{remark}

Next, we analyze the behavior of $\|\mathbb{E}[\bm {\tilde x}_{k}-\bm{x}_{LS}]\|_2$ obtained from the REK method, for solving the doubly noisy linear system \eqref{1.1}. To the best of our knowledge, this is the first work to analyze the REK method for solving doubly noisy linear system.
\begin{theorem}\label{thm3.2}
	Let $\bm {\tilde x}_{k}$ be the $k$-th iterate of the REK method applied to the doubly noisy linear system \eqref{1.1}. Let $\bm {x}_{LS}=A^{\dagger}\bm b$ and $\bm {\tilde x}_{LS}=\widetilde{A}^{\dagger}\bm {\tilde b}$ be the smallest norm least squares solution of the noiseless linear system \eqref{1.2} and the doubly-noisy linear system \eqref{1.1}, respectively. For any $\bm {\tilde x}_{0}\in \mathbb{R}^{n}$ and iteration number $k \in \mathbb{N}^+$, we have
{\small
\begin{align*}
	\|\mathbb{E} (\bm {\tilde x}_{k}-\bm {x}_{LS})\|_2\leq \alpha^{k}\left\|(\bm {\tilde x}_0-\bm {x}_{LS})_{\mathcal {R}(\widetilde{A}^\top)}\right\|_2+k\alpha^{k-1}\frac{\|\widetilde{A}^\top\bm {\tilde b}\|_2}{\|\widetilde{A}\|^2_F} &+\beta\left\|(\bm {x}_{LS}-\bm {\tilde x}_{LS})_{\mathcal {R}(\widetilde{A}^\top)}\right\|_2\\
	&+\left\|(\bm {\tilde x}_0-\bm {x}_{LS})_{\mathcal {N}(\widetilde{A})}\right\|_2,\nonumber
\end{align*}}
where $\alpha=1-\frac{{\sigma}^2_{\min}(\widetilde A)}{\|\widetilde{A}\|^2_F}$ and $\beta=1-\left(1-\frac{{\sigma}^2_{\max}(\widetilde A)}{\|\widetilde{A}\|^2_F}\right)^k$.
\end{theorem}
\begin{proof}
Applying the REK algorithm \eqref{rek} to the doubly-noisy linear system \eqref{1.1}, we have
{\small
\begin{align*}
	\mathbb{E}_{k-1}\left[\bm{\tilde z}_{k}\right]= \left( I -\sum_{j_{k}=1}^{n}\frac{\left\| \bm{\tilde{a}}_{j_{k}}\right\|_2^2}{\| \tilde{A} \|_F^2}\cdot \frac{\bm{\tilde{a}}_{j_k} \bm{\tilde{a}}^\top_{j_k}}{\|\bm{\tilde{a}}_{j_k}\|_2^2} \right) \bm{\tilde z}_{k-1}
	=\left(I-\frac{\widetilde{A}\widetilde{A}^\top}{\|\widetilde{A}\|^2_F}\right) \bm{\tilde z}_{k-1},
\end{align*}
\begin{align*}
\mathbb{E}\left[\bm{\tilde z}_{k}\right]=\mathbb{E}\left[\mathbb{E}_{k-1}\left[\bm{\tilde z}_{k}\right]\right]=\left(I-\frac{\widetilde{A}\widetilde{A}^\top}{\|\widetilde{A}\|^2_F}\right) \mathbb{E}\left[\bm{\tilde z}_{k-1}\right]=\cdots=\left(I-\frac{\widetilde{A}\widetilde{A}^\top}{\|\widetilde{A}\|^2_F}\right)^k \bm{\tilde z}_{0}
\end{align*}}
and	
{\small
\begin{align*}
\mathbb{E}_{k-1}\left[\bm{\tilde x}_{k} - \bm{x}_{LS}\right]	
	&=\bm{\tilde x}_{k-1} - \bm{x}_{LS} -\sum_{i_{k}=1}^{m}\frac{\left\| \bm{\tilde{a}}_{i_{k}}\right\|_2^2}{\| \tilde{A} \|_F^2} \frac{\bm{\tilde{a}}^\top_{i_k} \bm{\tilde x}_{k-1} - \tilde{b}_{i_k}+(\tilde z_{k-1})_{i_k}}{\|\bm{\tilde{a}}_{i_k}\|_2^2} \cdot \bm{\tilde{a}}_{i_k}\\
	&=\bm{\tilde x}_{k-1} - \bm{x}_{LS} -\frac{\widetilde{A}^\top(\widetilde{A}\bm{\tilde x}_{k-1} - \bm{\tilde b}+\bm{\tilde z}_{k-1})}{\|\widetilde{A}\|^2_F}\\	
	&=(\bm{\tilde x}_{k-1} - \bm{x}_{LS})-\frac{\widetilde{A}^\top\widetilde{A}\bm{\tilde x}_{k-1} - \widetilde{A}^\top\bm{\tilde b}}{\|\widetilde{A}\|^2_F}-\frac{\widetilde{A}^\top\bm{\tilde z}_{k-1}}{\|\widetilde{A}\|^2_F}\\
	&=(\bm{\tilde x}_{k-1} - \bm{x}_{LS})-\frac{\widetilde{A}^\top\widetilde{A}\bm{\tilde x}_{k-1} - \widetilde{A}^\top\widetilde{A}\bm{\tilde x}_{LS}}{\|\widetilde{A}\|^2_F}-\frac{\widetilde{A}^\top\bm{\tilde z}_{k-1}}{\|\widetilde{A}\|^2_F}\\
	&=\left(I-\frac{\widetilde{A}^\top\widetilde{A}}{\|\widetilde{A}\|^2_F}\right)(\bm{\tilde x}_{k-1} - \bm{x}_{LS})-\frac{\widetilde{A}^{T}\widetilde{A}(\bm{x}_{LS} - \bm {\tilde x}_{LS})}{\|\widetilde{A}\|^2_F}-\frac{\widetilde{A}^\top\bm{\tilde z}_{k-1}}{\|\widetilde{A}\|^2_F}.
\end{align*}}
By the law of total expectation, we have
{\small
\begin{align*}
&\mathbb{E}\left[\bm{\tilde x}_{k} - \bm{x}_{LS}\right]=\mathbb{E}\left[\mathbb{E}_{k-1}\left[\bm{\tilde x}_{k} - \bm{x}_{LS}\right]\right]\\
&=\left(I-\frac{\widetilde{A}^\top\widetilde{A}}{\|\widetilde{A}\|^2_F}\right)\mathbb{E}\left[\bm{\tilde x}_{k-1} - \bm{x}_{LS}\right]-\frac{\widetilde{A}^{\top}\widetilde{A}(\bm{x}_{LS} - \bm {\tilde x}_{LS})}{\|\widetilde{A}\|^2_F}-\frac{\widetilde{A}^\top}{\|\widetilde{A}\|^2_F}\left(I-\frac{\widetilde{A}\widetilde{A}^\top}{\|\widetilde{A}\|^2_F}\right)^{k-1} \bm{\tilde z}_0\\
&=\left(I-\frac{\widetilde{A}^\top\widetilde{A}}{\|\widetilde{A}\|^2_F}\right)\mathbb{E}\left[\bm{\tilde x}_{k-1} - \bm{x}_{LS}\right]-\frac{\widetilde{A}^{\top}\widetilde{A}(\bm{x}_{LS} - \bm {\tilde x}_{LS})}{\|\widetilde{A}\|^2_F}-\left(I-\frac{\widetilde{A}^\top\widetilde{A}}{\|\widetilde{A}\|^2_F}\right)^{k-1} \frac{\widetilde{A}^\top\bm{\tilde z}_0}{\|\widetilde{A}\|^2_F}\\
&=\left(I-\frac{\widetilde{A}^\top\widetilde{A}}{\|\widetilde{A}\|^2_F}\right)^2\mathbb{E}\left[\bm{\tilde x}_{k-2} - \bm{x}_{LS}\right]-\left(I-\frac{\widetilde{A}^\top\widetilde{A}}{\|\widetilde{A}\|^2_F}\right)\frac{\widetilde{A}^{\top}\widetilde{A}(\bm{x}_{LS} - \bm {\tilde x}_{LS})}{\|\widetilde{A}\|^2_F}\\
&\quad\quad -\frac{\widetilde{A}^{\top}\widetilde{A}(\bm{x}_{LS} - \bm {\tilde x}_{LS})}{\|\widetilde{A}\|^2_F}-2\left(I-\frac{\widetilde{A}^\top\widetilde{A}}{\|\widetilde{A}\|^2_F}\right)^{k-1} \frac{\widetilde{A}^\top\bm{\tilde z}_0}{\|\widetilde{A}\|^2_F}\\
&=\left(I-\frac{\widetilde{A}^\top\widetilde{A}}{\|\widetilde{A}\|^2_F}\right)^k\mathbb{E}\left[\bm{\tilde x}_{0} - \bm{x}_{LS}\right]-\sum_{j=0}^{k-1}\left(I-\frac{\widetilde{A}^\top\widetilde{A}}{\|\widetilde{A}\|^2_F}\right)^j\frac{\widetilde{A}^{\top}\widetilde{A}(\bm{x}_{LS} - \bm {\tilde x}_{LS})}{\|\widetilde{A}\|^2_F}\\
&\quad\quad -k\left(I-\frac{\widetilde{A}^\top\widetilde{A}}{\|\widetilde{A}\|^2_F}\right)^{k-1} \frac{\widetilde{A}^\top\bm{\tilde z}_0}{\|\widetilde{A}\|^2_F}\\
&=\left(I-\frac{\widetilde{A}^\top\widetilde{A}}{\|\widetilde{A}\|^2_F}\right)^k\left(\bm{\tilde x}_{0} - \bm{x}_{LS}\right)-\left[I-\left(I-\frac{\widetilde{A}^\top\widetilde{A}}{\|\widetilde{A}\|^2_F}\right)^k\right](\bm{x}_{LS} - \bm {\tilde x}_{LS})\\
&\quad\quad -k\left(I-\frac{\widetilde{A}^\top\widetilde{A}}{\|\widetilde{A}\|^2_F}\right)^{k-1} \frac{\widetilde{A}^\top\bm{\tilde b}}{\|\widetilde{A}\|^2_F}.
\end{align*} }
Taking the 2-norm on both sides of the above equation yields
{\small	
\begin{align*}
\left\|\mathbb{E}\left[\bm{\tilde x}_{k} - \bm{x}_{LS}\right]\right\|_2 &\leq
\left\|\left(I-\frac{\widetilde{A}^\top\widetilde{A}}{\|\widetilde{A}\|^2_F}\right)^k\left(\bm{\tilde x}_{0} - \bm{x}_{LS}\right)\right\|_2+\left\|k\left(I-\frac{\widetilde{A}^\top\widetilde{A}}{\|\widetilde{A}\|^2_F}\right)^{k-1} \frac{\widetilde{A}^\top\bm{\tilde b}}{\|\widetilde{A}\|^2_F}\right\|_2\\
&\quad +\left\|\left[I-\left(I-\frac{\widetilde{A}^\top\widetilde{A}}{\|\widetilde{A}\|^2_F}\right)^k\right](\bm{x}_{LS} - \bm {\tilde x}_{LS})\right\|_2\\
&\leq\left(1-\frac{{\sigma}^2_{\min}(\widetilde A)}{\|\widetilde{A}\|^2_F}\right)^k\left\|\left(\bm{\tilde x}_{0} - \bm{x}_{LS}\right)_{\mathcal {R}(\widetilde{A}^\top)}\right\|_2+k\left(1-\frac{{\sigma}^2_{\min}(\widetilde A)}{\|\widetilde{A}\|^2_F}\right)^{k-1}\frac{\left\|\widetilde{A}^\top\bm{\tilde b}\right\|_2}{\|\widetilde{A}\|^2_F}\\
&\quad +\left[1-\left(1-\frac{{\sigma}^2_{\max}(\widetilde A)}{\|\widetilde{A}\|_F^2}\right)^k\right]\left\|(\bm{x}_{LS} - \bm {\tilde x}_{LS})_{\mathcal {R}(\widetilde{A}^\top)} \right\|_2+\left\|\left(\bm{\tilde x}_{0} - \bm{x}_{LS}\right)_{\mathcal {N}(\widetilde{A})}\right\|_2.
\end{align*}}

In summary,
\[
\begin{aligned}
	\|\mathbb{E} (\bm {\tilde x}_{k}-\bm {x}_{LS})\|_2\leq & \alpha^{k}\left\|(\bm {\tilde x}_0-\bm {x}_{LS})_{\mathcal {R}(\widetilde{A}^\top)}\right\|_2+k\alpha^{k-1}\frac{\|\widetilde{A}^\top\bm {\tilde b}\|_2}{\|\widetilde{A}\|^2_F}+\left\|(\bm {\tilde x}_0-\bm {x}_{LS})_{\mathcal {N}(\widetilde{A})}\right\|_2\\
	&+\beta\left\|(\bm {x}_{LS}-\bm {\tilde x}_{LS})_{\mathcal {R}(\widetilde{A}^\top)}\right\|_2,
\end{aligned}
\]
where $\alpha=1-\frac{{\sigma}^2_{\min}(\widetilde A)}{\|\widetilde{A}\|^2_F}$ and by $\beta=1-\left(1-\frac{{\sigma}^2_{\max}(\widetilde A)}{\|\widetilde{A}\|^2_F}\right)^k$.
\end{proof}

\begin{remark}\label{rema3.1}
{\it Compared with Theorem \ref{thm3.1}, the convergence estimate of the REK algorithm in Theorem \ref{thm3.2} contains an additional factor $k\alpha^{k-1}$, which affects its convergence behavior during the early stage of iterations. This also implies that the REK method may converge a little slower than the RK method for solving the doubly noisy linear system \eqref{1.1}.
Nevertheless, as the iteration number $k$ increases, both the convergence rates of the REK method and the RK method are governed only by the quantity
\scalebox{0.9}{$\sigma_{\min}^2(\widetilde{A})/\|\widetilde{A}\|_F^2$}.
Moreover, the convergence horizon of REK coincides with that of RK, and is determined by
\scalebox{0.9}{$\big\|(\bm{x}_{LS}-\bm{\tilde{x}}_{LS})_{\mathcal{R}(\widetilde{A}^{\top})}\big\|_2+\big\|(\bm{\tilde{x}}_0-\bm{x}_{LS})_{\mathcal{N}(\widetilde{A})}\big\|_2$}.
This indicates that, for solving the doubly noisy linear system \eqref{1.1}, the RK method may be better than the REK method. Experiment \ref{sec5.2} effectively verify our theoretical findings.}
\end{remark}

\subsection{Convergence of RBK and RDBK for the doubly noisy linear system}\label{sec3.2}
In this section, we analyze the behavior of $\|\mathbb{E}[\bm {{\tilde x}}_{k}-\bm{x}_{LS}]\|_2$ for the RBK and RDBK methods applied to \eqref{1.1} directly. 
These two methods are two popular block Kaczmarz methods that are for consistent and inconsistent linear systems, respectively.
\begin{theorem}\label{thm3.3}
	Let $\bm{\tilde{x}}_{k}$ be the $k$-th iteration obtained from the RBK method applied to the doubly-noisy linear system \eqref{1.1}. Let
$\bm{x}_{LS}=A^{\dagger}\bm{b}$ and
$\bm{\tilde{x}}_{LS}=\widetilde{A}^{\dagger}\bm{\tilde{b}}$
be the smallest norm least squares solution of the noiseless system \eqref{1.2} and the doubly noisy system \eqref{1.1}, respectively. Consider the row partition \scalebox{0.9}{$\widetilde{A}=(\widetilde{A}_{\tau_1}^\top, \widetilde{A}_{\tau_2}^\top,\dots,\widetilde{A}_{\tau_p}^\top)^\top$}, with row paving $(p,\alpha_1,\alpha_2)$. Let $\widetilde{A}_{\tau_i}=U_i\Sigma_iV_i^\top$ be the full singular value decomposition of the row block $\widetilde{A}_{\tau_i}$, where {\small $V_i^\top=\underset{r_i}{\big[V_{r_i}} \bigm| \underset{n-r_i}{V_{n-r_i}\big]}$} is the corresponding partition of $V_i^{\top}$, and $r_i$ is the rank of $\widetilde{A}_{\tau_i}$. If the matrix 
{\small$$
\widetilde{A}_{\tau}=\frac{1}{\|\tilde{A} \|_F}(\|\widetilde{A}_{\tau_1}\|_F V_{r_1}, \|\widetilde{A}_{\tau_2}\|_F V_{r_2},\ldots,\|\widetilde{A}_{\tau_p}\|_F V_{r_p})\in \mathbb{R}^{n\times{\sum_{i=1}^{p} r_i}}
$$} is of full row rank, then for any initial point $\bm {\tilde x}_0\in \mathbb{R}^{n}$ and any iteration number
	$k\in\mathbb{N}^{+}$, we have
\begin{align*}
		\|\mathbb{E}(\bm {\tilde x}_{k}-\bm {x}_{LS})\|_2\leq (1-\eta)^k\left\|\left[\bm{\tilde x}_{0} - \bm{x}_{LS}\right]\right\|_2+\frac{\left[\cos\theta_{\gamma}\cdot\left(\left\| \bm {\tilde x}_{LS}\right\|_2+\frac{\left\|\bm {\tilde r}\right\|_2}{\sqrt{\alpha_1}}\right)+\left\|\bm {e}_{LS}\right\|_2\right]}{\eta} ,
\end{align*}
where 
\begin{equation}
\cos \theta_{\gamma}=\max_{\substack{ {i_k}\in \{1, 2,\dots,p\}}}\! \cos\angle(\bm {\tilde x}_{LS}-\bm {\tilde x}_{LS}^{(i_k)}, \mathcal {R}(\widetilde{A}_{\tau_{i_k}}^\top)),
\end{equation}
and $\eta=\lambda_{\min}(\widetilde{A}_{\tau}\widetilde{A}_{\tau}^\top)$, $\bm {\tilde r}=\widetilde {\bm {b}}-\widetilde{A} \bm {\tilde x}_{LS}$, $\bm {e}_{LS}=\bm{x}_{LS}-\bm {\tilde x}_{LS}$.
\end{theorem}

\begin{proof}
At the $k$-th iteration of the RBK method, an index set $\tau_{i_k} \in \{\tau_1, \tau_2,\ldots,\tau_p\}$ is randomly selected with the probability {\small $p_{i_k}=\big\|\widetilde{A}_{\tau_{i_k}}\big\|_F^2/\big\|\tilde{A}\big\|_F^2$.} Applying the RBK method to the doubly-noisy linear system \eqref{1.1}, it follows from \eqref{rbk} that
$$\bm{\tilde x}_{k} = \bm{\tilde x}_{k-1} + \widetilde{A}_{\tau_{i_k}}^\dagger (\bm{\tilde{b}}_{\tau_{i_k}}- \widetilde{A}_{\tau_{i_k}} \bm{\tilde x}_{k-1})$$
and
\[
\thickmuskip=1mu \medmuskip=1mu 
\begin{aligned}
	  \mathbb{E}\left[\bm{\tilde x}_k - \bm{x}_{LS}\right] &= (I-\sum_{i_k=1}^{p}p_{i_k} \widetilde{A}_{\tau_{i_k}}^\dagger \widetilde{A}_{\tau_{i_k}})\mathbb{E}\left[\bm{\tilde x}_{k-1} - \bm{x}_{LS}\right] -\sum_{i_k=1}^{p}p_{i_k} \widetilde{A}_{\tau_{i_k}}^\dagger (\widetilde{A}_{\tau_{i_k}} \bm{x}_{LS}-\bm{\tilde{b}}_{\tau_{i_k}}).
\end{aligned}
\]
Let
$$
S=\sum_{i_k=1}^{p}p_{i_k}\widetilde{A}_{\tau_{i_k}}^\dagger \widetilde{A}_{\tau_{i_k}}=\widetilde{A}_{\tau}\widetilde{A}_{\tau}^\top.
$$ 
Since $I \succeq \widetilde{A}_{\tau_{i_k}}^\dagger \widetilde{A}_{\tau_{i_k}}\succeq O$, we have $I \succeq S=\sum_{i_k=1}^{p}p_{i_k}\widetilde{A}_{\tau_{i_k}}^\dagger \widetilde{A}_{\tau_{i_k}} \succeq O$. As $\widetilde{A}_{\tau}$ has full row rank, $S$ is symmetric positive definite and $1 \geq \eta=\lambda_{\min}(S)>0$, and $\eta=1$ if and only if $S=I$. We have that
{\thickmuskip=1mu \medmuskip=1mu
\begin{align}\label{eqthm3.3}
\left\|\mathbb{E}\left[\bm{\tilde x}_k - \bm{x}_{LS}\right] \right\|_2 	&\leq \left\|(I-S)\mathbb{E}\left[\bm{\tilde x}_{k-1} - \bm{x}_{LS}\right]\right\|_2 + \left\|\sum_{i_k=1}^{p}p_{i_k} \widetilde{A}_{\tau_{i_k}}^\dagger (\widetilde{A}_{\tau_{i_k}} \bm{x}_{LS}-\bm{\tilde{b}}_{\tau_{i_k}})\right\|_2 \nonumber\\
&= \left\|(I-S)\mathbb{E}\left[\bm{\tilde x}_{k-1} - \bm{x}_{LS}\right]\right\|_2+ \left\|\sum_{i_k=1}^{p}p_{i_k} \widetilde{A}_{\tau_{i_k}}^\dagger \widetilde{A}_{\tau_{i_k}} (\bm{x}_{LS}-\bm {\tilde x}^{(i_k)}_{LS})\right\|_2 \nonumber\\
&\leq  \left\|(I-S)\right\|_2\left\|\mathbb{E}\left[\bm{\tilde x}_{k-1} - \bm{x}_{LS}\right]\right\|_2+\sum_{i_k=1}^{p}p_{i_k}\left\|\widetilde{A}_{\tau_{i_k}}^\dagger \widetilde{A}_{\tau_{i_k}} (\bm{x}_{LS}-\bm {\tilde x}^{(i_k)}_{LS})\right\|_2 \nonumber\\
&=(1-\eta)\left\|\mathbb{E}\left[\bm{\tilde x}_{k-1} - \bm{x}_{LS}\right]\right\|_2+\sum_{i_k=1}^{p}p_{i_k}\left\|\widetilde{A}_{\tau_{i_k}}^\dagger \widetilde{A}_{\tau_{i_k}} (\bm{x}_{LS}-\bm {\tilde x}^{(i_k)}_{LS})\right\|_2.
\end{align}}
Denote by $\theta_{i_k}=\angle(\bm {\tilde x}_{LS}-\bm {\tilde x}^{(i_k)}_{LS}, \mathcal {R}(\widetilde{A}_{\tau_{i_k}}^\top))$, then it follows from \cite[Definition 1.1 (Row paving)]{Need-2014} that 
\begin{align}\label{eqthm3.3.2}
\sum_{i_k=1}^{p}p_{i_k}& \left\|\widetilde{A}_{\tau_{i_k}}^\dagger \widetilde{A}_{\tau_{i_k}} (\bm{x}_{LS}-\bm {\tilde x}^{(i_k)}_{LS})\right\|_2 \leq \max_{{i_k}\in \{1, 2,\dots,p\}}\!\left\|\widetilde{A}_{\tau_{i_k}}^\dagger \widetilde{A}_{\tau_{i_k}} (\bm{x}_{LS}-\bm{\tilde x}^{(i_k)}_{LS})\right\|_2 \nonumber\\
\leq& \max_{ {i_k}\in \{1, 2,\dots,p\}}\!\left(\left\|\widetilde{A}_{\tau_{i_k}}^\dagger \widetilde{A}_{\tau_{i_k}} (\bm{x}_{LS}-\bm {\tilde x}_{LS})\right\|_2+\left\|\widetilde{A}_{\tau_{i_k}}^\dagger \widetilde{A}_{\tau_{i_k}} (\bm {\tilde x}_{LS}-\bm{\tilde x}_{LS}^{(i_k)})\right\|_2\right) \nonumber\\
\leq&\max_{ {i_k}\in \{1, 2,\dots,p\}}\!\left(\left\|\bm{x}_{LS}-\bm {\tilde x}_{LS}\right\|_2+ \cos\theta_{i_k}\cdot\left\| \bm {\tilde x}_{LS}-\bm{\tilde x}^{(i_k)}_{LS}\right\|_2\right) \nonumber\\
=& \max_{ {i_k}\in \{1, 2,\dots,p\}}\!\left(\left\|\bm{x}_{LS}-\bm {\tilde x}_{LS}\right\|_2+ \cos\theta_{i_k}\cdot\left\|\bm {\tilde x}_{LS}-\widetilde{A}_{\tau_{i_k}}^\dagger(\widetilde{A}_{\tau_{i_k}}\bm {\tilde x}_{LS}+\bm {\tilde r}_{\tau_{i_k}})\right\|_2\right) \nonumber\\
\leq& \max_{ {i_k}\in \{1, 2,\dots,p\}}\!\left(\left\|\bm{x}_{LS}-\bm {\tilde x}_{LS}\right\|_2+ \cos\theta_{i_k}\cdot\left(\left\| \bm {\tilde x}_{LS}\right\|_2+\left\|\widetilde{A}_{\tau_{i_k}}^\dagger\bm {\tilde r}_{\tau_{i_k}}\right\|_2\right)\right) \nonumber\\
\leq& \left\|\bm{x}_{LS}-\bm {\tilde x}_{LS}\right\|_2+ \cos\theta_{\gamma}\cdot\left(\left\| \bm {\tilde x}_{LS}\right\|_2+\frac{\left\|\bm {\tilde r}\right\|_2}{\sqrt{\alpha_1}}\right),
\end{align}
where we use
$$
\left\|\widetilde{A}_{\tau_{i_k}}^\dagger\bm {\tilde r}_{\tau_{i_k}}\right\|_2 \leq \sigma_{\max}(\widetilde{A}_{\tau_{i_k}}^\dagger) \left\|\bm {\tilde r}_{\tau_{i_k}}\right\|_2 \leq \frac{\bm {\left\|\tilde r}\right\|_2 }{\sigma_{\min}(\widetilde{A}_{\tau_{i_k}}) } \leq \frac{\left\|\bm {\tilde r}\right\|_2}{\sqrt{\alpha_1}}.
$$

Let $\bm {e}_{LS}=\bm{x}_{LS}-\bm {\tilde x}_{LS}$, a combination of \eqref{eqthm3.3} and \eqref{eqthm3.3.2} gives
{\thickmuskip=1mu \medmuskip=1mu 
\begin{align*}
&\left\|\mathbb{E}\left[\bm{\tilde x}_k - \bm{x}_{LS}\right] \right\|_2 \leq (1-\eta)\left\|\mathbb{E}\left[\bm{\tilde x}_{k-1} - \bm{x}_{LS}\right]\right\|_2+\cos\theta_{\gamma}\cdot\left(\left\| \bm {\tilde x}_{LS}\right\|_2+\frac{\left\|\bm {\tilde r}\right\|_2}{\sqrt{\alpha_1}}\right)+\left\|\bm {e}_{LS}\right\|_2\\
	& \leq (1-\eta)^2\left\|\mathbb{E}\left[\bm{\tilde x}_{k-2} - \bm{x}_{LS}\right]\right\|_2+(1-\eta) \left[\cos\theta_{\gamma}\cdot\left(\left\| \bm {\tilde x}_{LS}\right\|_2+\frac{\left\|\bm {\tilde r}\right\|_2}{\sqrt{\alpha_1}}\right)+\left\|\bm {e}_{LS}\right\|_2 \right]\\
	& \qquad+\cos\theta_{\gamma}\cdot\left(\left\| \bm {\tilde x}_{LS}\right\|_2+\frac{\left\|\bm {\tilde r}\right\|_2}{\sqrt{\alpha_1}}\right)+ \left\|\bm {e}_{LS}\right\|_2\\
	& \leq (1-\eta)^k\left\|\left[\bm{\tilde x}_{0} - \bm{x}_{LS}\right]\right\|_2+\sum_{j=1}^{k-1}(1-\eta)^j \left[ \cos\theta_{\gamma}\cdot\left(\left\| \bm {\tilde x}_{LS}\right\|_2+\frac{\left\|\bm {\tilde r}\right\|_2}{\sqrt{\alpha_1}}\right)+\left\|\bm {e}_{LS}\right\|_2\right]\\
	& \leq (1-\eta)^k\left\|\left[\bm{\tilde x}_{0} - \bm{x}_{LS}\right]\right\|_2+\frac{\left[\cos\theta_{\gamma}\cdot\left(\left\| \bm {\tilde x}_{LS}\right\|_2+\frac{\left\|\bm {\tilde r}\right\|_2}{\sqrt{\alpha_1}}\right)+\left\|\bm {e}_{LS}\right\|_2\right]}{\eta},
\end{align*}
which completes the proof.}
\end{proof}

Now, we are ready to analyze the behavior of $\|\mathbb{E}[\bm {\tilde x}_{k}-\bm{x}_{LS}]\|_2$ of the RDBK method for solving the doubly noisy linear system \eqref{1.1}.
\begin{theorem}\label{thm3.4}
	Let $\bm{\tilde{x}}_{k}$ be the $k$-th iteration obtained from the RDBK method applied to the doubly-noisy linear system \eqref{1.1}. Let
$\bm{x}_{LS}=A^{\dagger}\bm{b}$ and
$\bm{\tilde{x}}_{LS}=\widetilde{A}^{\dagger}\bm{\tilde{b}}$
be the smallest norm least squares solution of the noiseless system \eqref{1.2} and the doubly noisy system \eqref{1.1}, respectively. Consider the row partition \scalebox{0.9}{$\widetilde{A}=(\widetilde{A}_{\tau_1}^\top, \widetilde{A}_{\tau_2}^\top,\dots,\widetilde{A}_{\tau_p}^\top)^\top$}, with row paving $(p,\alpha_1,\alpha_2)$. Let $\widetilde{A}_{\tau_i}=U_i\Sigma_iV_i^\top$ be the full singular value decomposition of the row block $\widetilde{A}_{\tau_i}$, where {\small $V_i^\top=\underset{r_i}{\big[V_{r_i}} \bigm| \underset{n-r_i}{V_{n-r_i}\big]}$} is the corresponding partition of $V_i^{\top}$, and $r_i$ is the rank of $\widetilde{A}_{\tau_i}$. If the matrix 
{\small$$
\widetilde{A}_{\tau}=\frac{1}{\|\tilde{A} \|_F}(\|\widetilde{A}_{\tau_1}\|_F V_{r_1}, \|\widetilde{A}_{\tau_2}\|_F V_{r_2},\ldots,\|\widetilde{A}_{\tau_p}\|_F V_{r_p})\in \mathbb{R}^{n\times{\sum_{i=1}^{p} r_i}}
$$} is of full row rank, then for any initial point $\bm {\tilde x}_0\in \mathbb{R}^{n}$ and any iteration number
	$k\in\mathbb{N}^{+}$, we have
\begin{equation*}
	\|\mathbb{E} (\bm {\tilde x}_{k}\!-\!\bm {x}_{LS})\|_2\leq \!(1-\eta)^k\! \left\|\bm{\tilde x}_{0} \!-\! \bm{x}_{LS}\right\|_2\!+\!\frac{ \left[\cos\theta_{\gamma}\cdot\left(\left\| \bm {\tilde x}_{LS}\right\|_2+\frac{\left\|\bm {\tilde r}\right\|_2}{\sqrt{\alpha_1}}\right)\!+\!\left\|\bm {e}_{LS}\right\|_2\!+\!\frac{\left\|\bm{\tilde b}\right\|_2}{\sqrt{\alpha_1}} \right]}{\eta},
\end{equation*}
where 
\begin{equation}
\cos \theta_{\gamma}=\max_{\substack{ {i_k}\in \{1, 2,\dots,p\}}}\! \cos\angle(\bm {\tilde x}_{LS}-\bm {\tilde x}_{LS}^{(i_k)}, \mathcal {R}(\widetilde{A}_{\tau_{i_k}}^\top)),
\end{equation}
and $\eta=\lambda_{\min}(\widetilde{A}_{\tau}\widetilde{A}_{\tau}^\top)$, $\bm {\tilde r}=\widetilde {\bm {b}}-\widetilde{A} \bm {\tilde x}_{LS}$, $\bm {e}_{LS}=\bm{x}_{LS}-\bm {\tilde x}_{LS}$.
\end{theorem}

\begin{proof}
Let $\widetilde{A}=(\widetilde{A}_{\nu_1}, \widetilde{A}_{\nu_2},\dots,\widetilde{A}_{\nu_q})$ be a column partition of $\widetilde{A}$. Applying the RDBK method to the doubly-noisy linear system \eqref{1.1}, it follows from \eqref{rdbk} that
\begin{align}\label{eqn410}
\left\|\mathbb{E}\left[\bm{\tilde z}_{k}\right]\right\|_2	&=\left\|\left(I-\sum_{{i_k}=1}^{q}\frac{\left\|\widetilde{A}_{\nu_{i_k}}\right\|_F^2}{\| \tilde{A} \|_F^2}\widetilde{A}_{\nu_{i_k}}\widetilde{A}_{\nu_{i_k}}^\dagger\right) \mathbb{E}\left[\bm{\tilde z}_{k-1}\right]\right\|_2\nonumber\\
	& \leq \left\|I-\sum_{{i_k}=1}^{q}\frac{\left\|\widetilde{A}_{\nu_{i_k}}\right\|_F^2}{\| \tilde{A} \|_F^2}\widetilde{A}_{\nu_{i_k}}\widetilde{A}_{\nu_{i_k}}^\dagger\right\|_2\left\|\mathbb{E}\left[\bm{\tilde z}_{k-1}\right]\right\|_2\nonumber\\
	& \leq \left\|I-\frac{\widetilde{A}\widetilde{A}^\top}{\|\widetilde{A}\|^2_F}\right\|_2\left\|\mathbb{E}\left[\bm{\tilde z}_{k-1}\right]\right\|_2\nonumber\\
	& \leq\left\|\mathbb{E}\left[\bm{\tilde z}_{k-1}\right]\right\|_2\leq\cdots \nonumber\\
	& \leq\left\|\bm{\tilde z}_{0}\right\|_2,
\end{align}
where we used \scalebox{0.9}{$\sum_{{i_k}=1}^{q}\frac{\left\|\widetilde{A}_{\nu_{i_k}}\right\|_F^2}{\| \tilde{A} \|_F^2}\widetilde{A}_{\nu_{i_k}}\widetilde{A}_{\nu_{i_k}}^\dagger \succcurlyeq{  \frac{\widetilde{A}\widetilde{A}^\top}{\|\widetilde{A}\|^2_F}}$} in the third step.

At the $k$-th iteration of the RDBK method, an index $\tau_{i_k} \in \{\tau_1, \tau_2,\ldots,\tau_p\}$ is randomly selected with the probability {\small $p_{i_k}=\left\|\widetilde{A}_{\tau_{i_k}}\right\|_F^2 / \| \tilde{A} \|_F^2$}. Denote by $S=\widetilde{A}_{\tau}\widetilde{A}_{\tau}^\top$, by 
 $$
 M=(\|\widetilde{A}_{\tau_1}\|_F^2\widetilde{A}_{\tau_1}^{\dagger}, \|\widetilde{A}_{\tau_2}\|_F^2 \widetilde{A}_{\tau_2}^{\dagger},\ldots,\|\widetilde{A}_{\tau_p}\|_F^2 \widetilde{A}_{\tau_p}^{\dagger}),
 $$
 and by 
 $$
 \widetilde{A}_{\tau}=\frac{1}{\|\tilde{A} \|_F}(\|\widetilde{A}_{\tau_1}\|_F V_{r_1}, \|\widetilde{A}_{\tau_2}\|_F V_{r_2},\ldots,\|\widetilde{A}_{\tau_p}\|_F V_{r_p}),
 $$
 then we have
{\thickmuskip=1mu \medmuskip=1mu 
\begin{align*}
&\mathbb{E}_{k-1}\left[\bm{\tilde x}_{k} - \bm{x}_{LS}\right]	=\mathbb{E}_{k-1}\left[\mathbb{E}_{k-1}^i\left[\bm{\tilde x}_{k} - \bm{x}_{LS}\right]\right]\\
	&=\mathbb{E}_{k-1}\left[\mathbb{E}_{k-1}^i\left[\bm{\tilde x}_{k-1}- \bm{x}_{LS} + \widetilde{A}_{\tau_{i_k}}^\dagger (\bm{\tilde{b}}_{\tau_{i_k}}-(\bm{\tilde z}_k)_{ {\tau_{i_k}}}-\widetilde{A}_{\tau_{i_k}} \bm{\tilde x}_{k-1})\right]\right]\\
	&= \mathbb{E}_{k-1}\left[(I-\sum_{{i_k}=1}^{p}p_{i_k}\widetilde{A}_{\tau_{i_k}}^\dagger\widetilde{A}_{\tau_{i_k}})(\bm{\tilde x}_{k-1}- \bm{x}_{LS})+\sum_{{i_k}=1}^{p}p_{i_k}\widetilde{A}_{\tau_{i_k}}^\dagger (\bm{\tilde{b}}_{\tau_{i_k}}-(\bm{\tilde z}_k)_{ {\tau_{i_k}}}- \widetilde{A}_{\tau_{i_k}}\bm{x}_{LS})\right] \\
	&= \mathbb{E}_{k-1}\left[(I-\sum_{{i_k}=1}^{p}p_{i_k}V_{r_{i_k}}V_{r_{i_k}}^\top)(\bm{\tilde x}_{k-1}- \bm{x}_{LS})+\sum_{{i_k}=1}^{p}p_{i_k}\widetilde{A}_{\tau_{i_k}}^\dagger (\bm{\tilde{b}_{\tau_{i_k}}}-(\bm{\tilde z}_k)_{ {\tau_{i_k}}}- \widetilde{A}_{\tau_{i_k}} \bm{x}_{LS})\right] \\
	&= \left(I-S\right)(\bm{\tilde x}_{k-1}- \bm{x}_{LS})+\sum_{{i_k}=1}^{p}p_{i_k}\widetilde{A}_{\tau_{i_k}}^\dagger (\bm{\tilde{b}}_{\tau_{i_k}}- \widetilde{A}_{\tau_{i_k}} \bm{x}_{LS})-\mathbb{E}_{k-1}\left[\sum_{{i_k}=1}^{p}p_{i_k}\widetilde{A}_{\tau_{i_k}}^\dagger(\bm{\tilde z}_k)_{{\tau_{i_k}}}\right] \\
	&=\left(I-S\right)(\bm{\tilde x}_{k-1}- \bm{x}_{LS})+\sum_{{i_k}=1}^{p}p_{i_k}\widetilde{A}_{\tau_{i_k}}^\dagger (\bm{\tilde{b}_{\tau_{i_k}}}- \widetilde{A}_{\tau_{i_k}} \bm{x}_{LS}) -\frac{\mathbb{E}_{k-1}\left[M\bm{\tilde z}_k\right]}{\| \tilde{A} \|_F^2}.
\end{align*}
}
From the law of total expectation, we obtain
\begin{align*}
	\mathbb{E}\left[\bm{\tilde x}_{k} -\bm{x}_{LS}\right]= \left(I-S\right)\mathbb{E}\left[\bm{\tilde x}_{k-1} -\bm{x}_{LS}\right]+\sum_{{i_k}=1}^{p}p_{i_k}\widetilde{A}_{\tau_{i_k}}^\dagger (\bm{\tilde{b}}_{\tau_{i_k}}- \widetilde{A}_{\tau_{i_k}} \bm{x}_{LS})-\frac{M\cdot\mathbb{E}\left[\bm{\tilde z_k}\right]}{\| \tilde{A} \|_F^2} .
\end{align*}
Since $I \succeq \widetilde{A}_{\tau_{i_k}}^\dagger \widetilde{A}_{\tau_{i_k}}\succeq O$, we have $I \succeq S=\sum_{i_k=1}^{p}p_{i_k}\widetilde{A}_{\tau_{i_k}}^\dagger \widetilde{A}_{\tau_{i_k}} \succeq O$. As $\widetilde{A}_{\tau}$ has full row rank, $S$ is symmetric positive definite and $1 \geq \eta=\lambda_{\min}(S)>0$, and $\eta=1$ if and only if $S=I$. Taking norms in the above equation, we get from \eqref{eqn410} and \eqref{eqthm3.3.2} that
{\thickmuskip=1mu \medmuskip=0.8mu 
\begin{align*}
&\left\|\mathbb{E}\left[\bm{\tilde x}_{k}\!-\!\bm{x}_{LS}\right]\right\|_2\leq  \left\|\left(I\!-\!S\right)\mathbb{E}\left[\bm{\tilde x}_{k-1} \!-\!\bm{x}_{LS}\right]\right\|_2 \!+\! \left\|\sum_{{i_k}=1}^{p}p_{i_k}\widetilde{A}_{\tau_{i_k}}^\dagger (\bm{\tilde{b}}_{\tau_{i_k}}\!-\!\widetilde{A}_{\tau_{i_k}}\bm{x}_{LS})\right\|_2\!+\! \frac{\left\|M\mathbb{E}\left[\bm{\tilde z}_k\right]\right\|_2}{\| \tilde{A}\|_F^2} \\
&\leq  (1-\eta)\left\|\mathbb{E}\left[\bm{\tilde x}_{k-1} - \bm{x}_{LS}\right]\right\|_2+\sum_{i_k=1}^{p}p_{i_k} \left\|\widetilde{A}_{\tau_{i_k}}^\dagger \widetilde{A}_{\tau_{i_k}} (\bm {\tilde x}^{(i_k)}_{LS}-\bm{x}_{LS})\right\|_2+\frac{\left\|M\right\|_2\left\|\bm{\tilde z}_{0}\right\|_2}{\| \tilde{A} \|_F^2}   \\
& \leq {\small (1-\eta)\left\|\mathbb{E}\left[\bm{\tilde x}_{k-1} - \bm{x}_{LS}\right]\right\|_2+\cos\theta_{\gamma}\cdot\left(\left\| \bm {\tilde x}_{LS}\right\|_2+\frac{\left\|\bm {\tilde r}\right\|_2}{\sqrt{\alpha_1}}\right)+\left\|\bm{x}_{LS}-\bm {\tilde x}_{LS}\right\|_2+\frac{\left\|M\right\|_2\left\|\bm{\tilde z}_{0}\right\|_2}{\| \tilde{A} \|_F^2}}.
\end{align*}}

Moreover, 
{\thickmuskip=1mu \medmuskip=1mu 
\begin{align*}
\frac{\left\|M\right\|_2\left\|\bm{\tilde z}_{0}\right\|_2}{\| \tilde{A} \|_F^2}&=\frac{\left\|\bm{\tilde z}_{0}\right\|_2}{\| \tilde{A} \|_F^2} \max_{\|\bm x\|_2=1}\!\left\|\sum_{i=1}^{p}\|\widetilde{A}_{\tau_i}\|_F^2\widetilde{A}_{\tau_i}^{\dagger} \bm{ x}_{\tau_i}\right\|_2\leq \frac{\left\|\bm{\tilde z}_{0}\right\|_2}{\| \tilde{A} \|_F^2} \max_{\|\bm x\|_2=1}\!\sum_{i=1}^{p} \left\|\|\widetilde{A}_{\tau_i}\|_F^2\widetilde{A}_{\tau_i}^{\dagger} \bm{ x}_{\tau_i}\right\|_2 \nonumber\\
&\leq \frac{\left\|\bm{\tilde z}_{0}\right\|_2}{\| \tilde{A} \|_F^2} \max_{\|\bm x\|_2=1}\!\left(\sum_{i=1}^{p} \left\|\|\widetilde{A}_{\tau_i}\|_F^2\widetilde{A}_{\tau_i}^{\dagger}\right\|_2 \cdot\left\|\bm{ x}_{\tau_i}\right\|_2 \right)\nonumber\\
& \leq \frac{\left\|\bm{\tilde z}_{0}\right\|_2}{\| \tilde{A} \|_F^2} \max_{\|\bm x\|_2=1}\!\left(\sqrt{\sum_{i=1}^{p} \left\|\|\widetilde{A}_{\tau_i}\|_F^2\widetilde{A}_{\tau_i}^{\dagger}\right\|_2^2}\cdot \sqrt{\sum_{i=1}^{p}\left\|\bm{ x}_{\tau_i}\right\|_2^2}\right)\nonumber\\
&=\frac{\left\|\bm{\tilde z}_{0}\right\|_2}{\| \tilde{A} \|_F^2} \sqrt{\sum_{i=1}^{p} \left\|\|\widetilde{A}_{\tau_i}\|_F^2\widetilde{A}_{\tau_i}^{\dagger}\right\|_2^2}\nonumber\\
&\leq \frac{\left\|\bm{\tilde z}_{0}\right\|_2}{\| \tilde{A} \|_F^2} \sqrt{\frac{\sum_{i=1}^{p} \|\widetilde{A}_{\tau_i}\|_F^4}{\alpha_1}}\leq\frac{\left\|\bm{\tilde z}_{0}\right\|_2}{\sqrt{\alpha_1}},
\end{align*}}
where the last step follows from \cite[Definition 1.1 (Row paving)]{Need-2014}.
Combining the above results altogether, we arrive at
{\small
{\thickmuskip=1mu \medmuskip=1mu 
\begin{align*}
&\left\|\mathbb{E}\left[\bm{\tilde x}_{k}\!-\!\bm{x}_{LS}\right]\right\|_2 \leq (1-\eta)\left\|\mathbb{E}\left[\bm{\tilde x}_{k-1} - \bm{x}_{LS}\right]\right\|_2+\cos\theta_{\gamma}\cdot\left(\left\| \bm {\tilde x}_{LS}\right\|_2+\frac{\left\|\bm {\tilde r}\right\|_2}{\sqrt{\alpha_1}}\right)+\left\|\bm{x}_{LS}-\bm {\tilde x}_{LS}\right\|_2+\frac{\left\|\bm{\tilde z}_{0}\right\|_2}{\sqrt{\alpha_1}}  \\
& \leq  (1-\eta)^2 \left\|\mathbb{E}\left[\bm{\tilde x}_{k-2} - \bm{x}_{LS}\right]\right\|_2\!+\cos\theta_{\gamma}\cdot\left(\left\| \bm {\tilde x}_{LS}\right\|_2+\frac{\left\|\bm {\tilde r}\right\|_2}{\sqrt{\alpha_1}}\right)+\left\|\bm{x}_{LS}-\bm {\tilde x}_{LS}\right\|_2+\frac{\left\|\bm{\tilde z}_{0}\right\|_2}{\sqrt{\alpha_1}} \\
&\quad+\!(1-\eta)\left[\cos\theta_{\gamma}\cdot\left(\left\| \bm {\tilde x}_{LS}\right\|_2\!+\!\frac{\left\|\bm {\tilde r}\right\|_2}{\sqrt{\alpha_1}}\right)\!+\!\left\|\bm{x}_{LS}-\bm {\tilde x}_{LS}\right\|_2\!+\!\frac{\left\|\bm{\tilde z}_{0}\right\|_2}{\sqrt{\alpha_1}} \right]  \\
&\leq (1-\eta)^k \left\|\mathbb{E}\left[\bm{\tilde x}_{0}\!-\! \bm{x}_{LS}\right]\right\|_2\!+\!\sum_{j=0}^{k-1}(1\!-\!\eta)^j \left[\cos\theta_{\gamma}\cdot\left(\left\| \bm {\tilde x}_{LS}\right\|_2\!+\!\frac{\left\|\bm {\tilde r}\right\|_2}{\sqrt{\alpha_1}}\right)\!+\!\left\|\bm{x}_{LS}-\bm {\tilde x}_{LS}\right\|_2\!+\!\frac{\left\|\bm{\tilde z}_{0}\right\|_2}{\sqrt{\alpha_1}} \right] \\
&\leq (1-\eta)^k \left\|\mathbb{E}\left[\bm{\tilde x}_{0} - \bm{x}_{LS}\right]\right\|_2\!+\!\frac{\left[\cos\theta_{\gamma}\cdot\left(\left\| \bm {\tilde x}_{LS}\right\|_2+\frac{\left\|\bm {\tilde r}\right\|_2}{\sqrt{\alpha_1}}\right)+\left\|\bm{x}_{LS}-\bm {\tilde x}_{LS}\right\|_2+\frac{\left\|\bm{\tilde z}_{0}\right\|_2}{\sqrt{\alpha_1}} \right]}{\eta},
\end{align*}}
}
which completes the proof.
\end{proof}

\begin{remark}\label{rema3.2}
{\it Two remarks are in order. Firstly, Theorems \ref{thm3.3} and \ref{thm3.4} demonstrate that both RBK and RDBK method converge to a ball centered at $\bm{x}_{LS}$ when solving the doubly noisy linear system \eqref{1.1}, provided that
$$
\widetilde{A}_{\tau}=\frac{1}{\|\widetilde{A}\|_F^2}(\|\widetilde{A}_{\tau_1}\|_F V_{r_1},\|\widetilde{A}_{\tau_2}\|_F V_{r_2},\ldots,\|\widetilde{A}_{\tau_p}\|_F V_{r_p})
$$
is of full row rank, and the condition is mild when $m \gg n$.
The convergence rates of RBK and RDBK are comparable, which are determined by \scalebox{0.9}{$\eta=\lambda_{\min}(\widetilde{A}_{\tau}\widetilde{A}_{\tau}^{\top})$}. Secondly, the convergence horizon of the RBK and RDBK methods is influenced by the angle \scalebox{0.9}{$\cos\angle(\bm {\tilde x}_{LS}-\bm {\tilde x}^{(i_k)}_{LS}, \mathcal {R}(\widetilde{A}_{\tau_{i_k}}^\top))$}, the scalars $\eta$ and $\alpha_1$, as well as $\left\|\bm{x}_{LS}-\bm {\tilde x}_{LS}\right\|_2$. However, the convergence horizon of the RDBK method contains an additional factor $\frac{\left\|\bm{\tilde z}_{0}\right\|_2}{\sqrt{\alpha_1}}$, which implies that the convergence horizon of the RDBK method is larger than that of the RBK method.
Thus, the RBK method will work better than the RDBK method when solving the doubly noisy linear system \eqref{1.1}. One refers to Experiment \ref{sec5.4} for more details.
}

\end{remark}

%
\section{On convergence of the sketch-and-project method for solving doubly noisy linear systems}\label{sec6}
In this section, we analyze the convergence behavior of the more general sketch-and-project method applied to the doubly noisy linear system.
Similarly, we consider the expectation of the error norm $\mathbb{E}\bigl[\|\bm{\tilde{x}}_k-\bm{x}_{LS}\|_B^2\bigr]$
as well as the norm of the expected error $\bigl\|\mathbb{E}[\bm{\tilde{x}}_k-\bm{x}_{LS}]\bigr\|_B$.
We first need the following lemma, whose proof is straightforward and thus is omitted.
\begin{lemma}\label{lem51}
Let $\widehat{G}\in\mathbb{R}^{n\times n}$ be a nonzero symmetric
matrix satisfying $O \preceq \widehat{G} \preceq I$ with $r=\operatorname{rank}(\widehat{G})$.
Let $B\in\mathbb{R}^{n\times n}$ be a symmetric positive definite matrix, and let $G=B^{-\frac{1}{2}}\widehat{G}B^{\frac{1}{2}}$.
Denote by $\lambda_r=\lambda_{\min}^{+}(\widehat{G})$ the smallest nonzero eigenvalue of $\widehat{G}$ and by $\lambda_1=\lambda_{\max}(\widehat{G})$ its largest eigenvalue. Then, for every $\bm{v}\in\mathcal{R}(G)$ and every positive integer
$k$, the following inequalities hold
\begin{equation}\label{b1}
\lambda_r\|\bm{v}\|_B^2
\leq
\bm{v}^{\top}BG\bm{v}
\leq
\lambda_1\|\bm{v}\|_B^2,
\end{equation}
\begin{equation}\label{b2}
\left\|\left(I-G\right)^k \bm {v}\right\|_B \leq \left(1-\lambda_r\right)^k \left\|\bm {v}\right\|_B,
\end{equation}
\begin{equation}\label{b3}
\left\|\left[I-\left(I-G\right)^k \right]\bm {v}\right\|_B \leq \left[1-\left(1-\lambda_1\right)^k\right] \left\|\bm {v}\right\|_B,
\end{equation}
and
\begin{equation}\label{b4}
\left\|\!\sum_{j=0}^{k-1}\left(I-G\right)^j \bm {v}\right\|_B \leq \frac{\left\|\bm {v}\right\|_B}{\lambda_r},
\end{equation}
\end{lemma}

Now, we are ready to analyze the behavior of $\mathbb{E}\|\bm {\tilde x}_{k}-\bm{x}_{LS}\|_2$ of the sketch-and-project method for solving the doubly noisy linear system \eqref{1.1}. The sketch-and-project method \cite{Gower2015} involves two parameters: a symmetric positive definite matrix $B$ defining the $B$-inner product, and a random sketch matrix $S_k$ sampled i.i.d at each iteration.
\begin{theorem}\label{th51}
Let $\bm {\tilde x}_{k}$ be the $k$-th iterate of the sketch-and-project method applied to the doubly noisy linear system \eqref{1.1}. Let $\bm{x}_{LS}=A^{\dagger}\bm{b}$ be the smallest norm least squares solution of the noiseless linear system \eqref{1.2}. Set $P_k=B^{-1}\widetilde{A}^\top H_k \widetilde{A}$ and $G=\mathbb{E}[P_k]$, where \scalebox{0.9}{$H_k=S_k\left(S_k^{\top}\widetilde{A}B^{-1}\widetilde{A}^{\top}S_k\right)^{\dagger}S_k^{\top}$}. For any initial point $\bm {\tilde x}_{0}\in \mathbb{R}^{n}$ and iteration number $k \in \mathbb{N}^+$, we have
\begin{align*}
\mathbb{E}\|\bm {\tilde x}_{k}-\bm {x}_{LS}\|_B^2\leq (1-\lambda_r)^{k}\left\|(\bm {\tilde x}_0-\bm {x}_{LS})_{\mathcal {R}(G)}\right\|_B^2&+\big\|(\bm {\tilde x}_0-\bm {x}_{LS})_{\mathcal {N}(G)}\big\|_B^2+\frac{\bm {r}^\top \mathbb{E}[H_k] \bm {r}}{\lambda_r},\nonumber
\end{align*}
where $\lambda_r$ is the smallest nonzero eigenvalue of the matrix $G$, and $\bm{r}=\bm{\tilde b}-\widetilde{A}\bm { x}_{LS}$.

\end{theorem}

\begin{proof}
Applying the sketch-and-project method \eqref{SP} to the doubly-noisy linear system \eqref{1.1}, we have
\begin{align}\label{51}
\bm {\tilde x}_{k}-\bm {x}_{LS}&=\bm{\tilde x}_{k-1}-\bm {x}_{LS}-B^{-1}\widetilde{A}^\top H_k\left(\widetilde{A}\bm{\tilde x}_{k-1}-\bm{b}\right) \nonumber\\
&=\left(I-B^{-1}\widetilde{A}^\top H_k \widetilde{A}\right)\left(\bm {\tilde x}_{k-1}-\bm {x}_{LS}\right)+B^{-1}\widetilde{A}^\top H_k(\bm{\tilde b}-\widetilde{A}\bm { x}_{LS}) \nonumber\\
&=\left(I-P_k\right)\left(\bm {\tilde x}_{k-1}-\bm {x}_{LS}\right)+B^{-1}\widetilde{A}^\top H_k(\bm{\tilde b}-\widetilde{A}\bm { x}_{LS}).
\end{align}
From \cite[Lemma 2.2]{Gower2015}, we have that $P_k$ projects orthogonally onto the subspace $\mathcal {R}(B^{-1}\widetilde{A}^\top S_k)$ and $I-P_k$ projects orthogonally onto the subspace $\mathcal {N}(S_k^\top \widetilde{A})$ with respect to the $B$-inner product. That is,
\begin{equation}\label{52}
\mathcal{R}(P_k)
=
\mathcal{R}(B^{-1}\widetilde{A}^{\top}S_k),
\qquad
\mathcal{R}(I-P_k)
=
\mathcal{N}(S_k^{\top}\widetilde{A}).
\end{equation}

%
First, we prove that $\mathcal {R}(G)$ and $\mathcal {N}(G)$ are orthogonal complement with respect to the $B$-inner product. 
Since $H_k^\top=H_k$ and $B^\top=B$, we have 
\begin{equation}\label{41}
P_k^\top B=B P_k=\widetilde{A}^\top H_k \widetilde{A}.
\end{equation}
Taking expectations yields $G^\top B=BG$. Thus, $G$ is self-adjoint with respect to the $B$-inner product, and $\widehat{G}=B^{\frac{1}{2}}GB^{-\frac{1}{2}}$ is symmetric. 
For any $\bm{v}\in\mathcal{R}(G)$, there exists
$\bm{y}\in\mathbb{R}^{n}$ such that $\bm{v}=G\bm{y}$. For
any $\bm{z}\in\mathcal{N}(G)$, we have
\[
\langle \bm{v},\bm{z}\rangle_B
=
\langle G\bm{y},\bm{z}\rangle_B
=
\langle \bm{y},G\bm{z}\rangle_B
=
0,
\]
which implies that $\mathcal{R}(G)$ and $\mathcal{N}(G)$ are orthogonal to each other with respect to the $B$-inner product. In addition, $\dim\bigl(\mathcal{R}(G)\bigr)+\dim\bigl(\mathcal{N}(G)\bigr)=n$, so $\mathcal {R}(G)$ and $\mathcal {N}(G)$ are orthogonal complement with respect to the $B$-inner product. 


Second, we show that the spaces $\mathcal {R}(P_k)$ and $\mathcal {N}(G)$ are orthogonal with respect to the $B$-inner product, and $\mathcal{R}(P_k)\subseteq \mathcal {R}(G)$.
We first prove that $\mathcal{N}(P_k)=\mathcal{N}(G)$. Indeed, on the one hand, for any $\bm{z}\in\mathcal{N}(G)$, we have 
$$\mathbb{E}\left\|P_k \bm{z}\right\|_B^2=\mathbb{E}\left[\bm{z}^\top P_k^\top B\bm{z}\right]=\mathbb{E}\left[\bm{z}^\top BP_k\bm{z}\right]=\bm{z}^\top BG\bm{z}=0$$
if and only if
$P_k \bm{z}=0$, and $\mathcal{N}(G)\subseteq\mathcal{N}(P_k)$. On the other hand, if $\bm{z}\in\mathcal{N}(P_k)$, then $G\bm{z}=\mathbb{E}\left[P_k \right]\bm{z}=0$, and $\mathcal{N}(P_k)\subseteq\mathcal{N}(G)$.

In addition, given $\bm{u}\in\mathcal{R}(P_k)$, there exists
$\bm{g}\in\mathbb{R}^{n}$ such that $\bm{u}=P_k\bm{g}$. For
any $\bm{z}\in\mathcal{N}(G)=\mathcal{N}(P_k)$, it follows from \eqref{41} that
\[\langle \bm{u},\bm{z}\rangle_B
=\langle P_k\bm{g},\bm{z}\rangle_B
=\langle \bm{g},P_k\bm{z}\rangle_B
=0.
\]
Thus, $\mathcal {R}(P_k)$ and $\mathcal {N}(G)$ are orthogonal to each other with respect to the $B$-inner product, and $\mathcal{R}(P_k)\subseteq \mathcal {R}(G)$.

Denote by $\bm{e}_k=\bm {\tilde x}_{k}-\bm {x}_{LS}$, it follows from \eqref{51} and \eqref{52} that 
\begin{equation}\label{b10}
\bm{e}_k=\bm{e}_{k-1}-P_k \bm{e}_{k-1}+B^{-1}\widetilde{A}^\top H_k(\bm{\tilde b}-\widetilde{A}\bm { x}_{LS}),
\end{equation}
where 
\begin{equation}\label{59}
-P_k \bm{e}_{k-1}+B^{-1}\widetilde{A}^\top H_k(\bm{\tilde b}-\widetilde{A}\bm { x}_{LS}) \in \mathcal {R}(P_k) \subseteq \mathcal {R}(G).
\end{equation}
Let $\bm{e} _k=(\bm{e}_k)_{\mathcal {R}(G)}+(\bm{e}_k)_{\mathcal {N}(G)}$
be the orthogonal decomposition with respect to the
$B$-inner product. From \eqref{b10}, \eqref{59} and the fact that $\mathcal {R}(G)$ and $\mathcal {N}(G)$ are orthogonal complement with respect to the $B$-inner product, we obtain
\begin{equation}\label{b11}
(\bm{e}_k)_{\mathcal {N}(G)}=(\bm{e}_{k-1})_{\mathcal {N}(G)}=\cdots=(\bm{e}_{0})_{\mathcal {N}(G)}.
\end{equation}
As a result,
\begin{equation}\label{b12}
\|\bm{e}_k\|_B^2=\|(\bm{e}_{k})_{\mathcal {R}(G)}\|_B^2+\|(\bm{e}_0)_{\mathcal {N}(G)}\|_B^2
\end{equation}
and
\begin{equation}\label{b13}
(\bm{e}_{k})_{\mathcal {R}(G)}=\left(I-P_k\right)(\bm{e}_{k-1})_{\mathcal {R}(G)}+B^{-1}\widetilde{A}^\top H_k(\bm{\tilde b}-\widetilde{A}\bm { x}_{LS}).
\end{equation}

Notice that
$B^{-1}\widetilde{A}^{\top}H_k \bigl(\widetilde{\bm{b}} -\widetilde{A}\bm{x}_{\mathrm{LS}}\bigr) \in \mathcal{R}(B^{-1}\widetilde{A}^{\top}S_k)
=\mathcal{R}(P_k)$, and $\mathcal{R}(P_k)$ and $\mathcal{R}(I-P_k)$ are orthogonal
with respect to the $B$-inner product,
the two terms on the right-hand side of \eqref{b13} are
$B$-orthogonal. Taking squared $B$-norm on both sides
of \eqref{b13}, we get
\begin{align}\label{b14}
\left\|(\bm{e}_k)_{\mathcal{R}(G)}\right\|_B^2=&\left\|(I-P_k)(\bm{e}_{k-1})_{\mathcal{R}(G)}\right\|_B^2
+\left\|B^{-1}\widetilde{A}^{\top}H_k\bigl(\widetilde{\bm{b}}-\widetilde{A}\bm{x}_{\mathrm{LS}}\bigr)\right\|_B^2\nonumber\\
&=\left\|(\bm{e}_{k-1})_{\mathcal{R}(G)}\right\|_B^2 -\left\|P_k\cdot(\bm{e}_{k-1})_{\mathcal{R}(G)}\right\|_B^2+\bm {r}^\top H_k\bm {r}\nonumber\\
&=\left\|(\bm{e}_{k-1})_{\mathcal{R}(G)}\right\|_B^2 -(\bm{e}_{k-1})_{\mathcal{R}(G)}BP_k(\bm{e}_{k-1})_{\mathcal{R}(G)} +\bm {r}^\top H_k\bm {r},
\end{align}
where $\bm{r}=\bm{\tilde b}-\widetilde{A}\bm { x}_{LS}$.
Taking expectations conditioned on $\bm{e}_{k-1}$ in \eqref{b14}, we get
\begin{align*}
\mathbb{E}_{k-1}\left\|(\bm{e}_k)_{\mathcal{R}(G)}\right\|_B^2 &=\left\|(\bm{e}_{k-1})_{\mathcal{R}(G)}\right\|_B^2 -(\bm{e}_{k-1})_{\mathcal{R}(G)}BG(\bm{e}_{k-1})_{\mathcal{R}(G)} +\bm {r}^\top \mathbb{E}[H_k]\bm {r}\\
& \leq (1-\lambda_r) \left\|(\bm{e}_{k-1})_{\mathcal{R}(G)}\right\|_B^2+\bm {r}^\top \mathbb{E}[H_k]\bm {r},
\end{align*}
where the last step is from \eqref{b1}. Taking expectation on both sides of the above inequality yields
\begin{align}\label{b15}
\mathbb{E}\left\|(\bm{e}_k)_{\mathcal{R}(G)}\right\|_B^2 
& \leq(1-\lambda_r) \mathbb{E}\left\|(\bm{e}_{k-1})_{\mathcal{R}(G)}\right\|_B^2+\bm {r}^\top \mathbb{E}[H_k]\bm {r}\nonumber\\
& \leq (1-\lambda_r)^k \mathbb{E}\left\|(\bm{e}_{0})_{\mathcal{R}(G)}\right\|_B^2+\!\sum_{j=0}^{k-1}(1\!-\!\lambda_r)^j \cdot\bm {r}^\top \mathbb{E}[H_k]\bm {r}\nonumber\\
& \leq (1-\lambda_r)^k \mathbb{E}\left\|(\bm{e}_{0})_{\mathcal{R}(G)}\right\|_B^2+\frac{\bm {r}^\top \mathbb{E}[H_k]\bm {r}}{\lambda_r}.
\end{align}
Combining \eqref{b12} and \eqref{b15}, we arrive at
\begin{align*}
\mathbb{E}\|\bm{e}_k\|_B^2&=\mathbb{E}\|(\bm{e}_{k})_{\mathcal {R}(G)}\|_B^2+\|(\bm{e}_0)_{\mathcal {N}(G)}\|_B^2\\
&\leq (1-\lambda_r)^k \mathbb{E}\left\|(\bm{e}_{0})_{\mathcal{R}(G)}\right\|_B^2+\frac{\bm {r}^\top \mathbb{E}[H_k]\bm {r}}{\lambda_r}+\|(\bm{e}_0)_{\mathcal {N}(G)}\|_B^2,
\end{align*}
which completes the proof.
\end{proof}
\begin{remark}
{\it Compared with Theorem \ref{thm3.5}, Theorem \ref{th51} completely eliminates the limitation arising from the convergence factor $\xi_{k-1}$. Moreover, no initial assumptions are imposed on Theorem \ref{th51}, and the result applies to both consistent and inconsistent systems with either full-rank or rank-deficient coefficient matrices.}
\end{remark}


Now, we are ready to analyze the behavior of $\|\mathbb{E}\left[\bm {\tilde x}_{k}-\bm{x}_{LS}\right]\|_2$ of the sketch-and-project method for solving the doubly noisy linear system \eqref{1.1}.

\begin{theorem}\label{th52}
Let $\bm {\tilde x}_{k}$ be the $k$-th iterate of the sketch-and-project method applied to the doubly noisy linear system \eqref{1.1}. Let $\bm {x}_{LS}=A^{\dagger}\bm b$ and $\bm {\tilde x}_{LS}=\widetilde{A}^{\dagger}\bm {\tilde b}$ be the smallest norm least squares solution of the noiseless linear system \eqref{1.2} and the doubly noisy linear system \eqref{1.1}, respectively. Set $P_k=B^{-1}\widetilde{A}^\top H_k \widetilde{A}$ and $G=\mathbb{E}[P_k]$, where \scalebox{0.9}{$H_k=S_k\left(S_k^{\top}\widetilde{A}B^{-1}\widetilde{A}^{\top}S_k\right)^{\dagger}S_k^{\top}$}. For any initial point $\bm {\tilde x}_{0}\in \mathbb{R}^{n}$ and iteration number $k \in \mathbb{N}^+$, we have
\begin{align*}
\|\mathbb{E}\left[\bm {\tilde x}_{k}-\bm {x}_{LS}\right]\|_B\leq &(1-\lambda_r)^{k}\left\|(\bm {\tilde x}_0-\bm {x}_{LS})_{\mathcal {R}(G)}\right\|_B+\big\|(\bm {\tilde x}_0-\bm {x}_{LS})_{\mathcal {N}(G)}\big\|_B\\
&+\left[1-(1-\lambda_1)^k\right]\left\|(\bm {\tilde x}_{LS}-\bm {x}_{LS})_{\mathcal {R}(G)}\right\|_B+\frac{\|B^{-1}\widetilde{A}^\top \mathbb{E}[H_k] \bm {\tilde r}\|_B}{\lambda_r},\nonumber
\end{align*}
where $\lambda_r$ and $\lambda_1$ are the smallest nonzero eigenvalue and the largest one of the matrix $G$, respectively, and $\bm{\tilde r}=\bm{\tilde b}-\widetilde{A}\bm {\tilde x}_{LS}$.

\end{theorem}

\begin{proof}
Applying the sketch-and-project method \eqref{SP} to the doubly-noisy linear system \eqref{1.1}, we have
\begin{align}\label{515}
\bm {\tilde x}_{k}-\bm {x}_{LS}=\left(I-P_k\right)\left(\bm {\tilde x}_{k-1}-\bm {x}_{LS}\right)+P_k(\bm {\tilde x}_{LS}-\bm {x}_{LS})+B^{-1}\widetilde{A}^\top H_k(\bm{\tilde b}-\widetilde{A}\bm {\tilde x}_{LS}).
\end{align}
From \cite[Lemma 2.2]{Gower2015}, it follows that $P_k$ projects orthogonally onto the subspace $\mathcal {R}(B^{-1}\widetilde{A}^\top S_k)$ and $I-P_k$ projects orthogonally onto the subspace $\mathcal {N}(S_k^\top \widetilde{A})$ with respect to the $B$-inner product. Let $\widehat{G}=B^{\frac{1}{2}}GB^{-\frac{1}{2}}$, we get from $O \preceq P_k \preceq I$ that 
$$
O \preceq G \preceq I \quad\text{and} \quad O \preceq \widehat{G} \preceq I.
$$ 
Denote by $\bm{e}_k=\bm {\tilde x}_{k}-\bm {x}_{LS}$ and by $\bm{d}=\bm {\tilde x}_{LS}-\bm {x}_{LS}$. 
Taking expectations conditioned on $\bm{\tilde x}_{k-1}$ in \eqref{515}, we get
$$
\mathbb{E}_{k-1}[\bm{e}_k]= (I-\mathbb{E}_{k-1}[P_k])\bm{e}_{k-1}+\mathbb{E}_{k-1}[P_k]\bm{d}+B^{-1}\widetilde{A}^\top \mathbb{E}_{k-1}[H_k](\bm{\tilde b}-\widetilde{A}\bm {\tilde x}_{LS}).
$$
Taking expectation again gets
\begin{align}\label{516}
\mathbb{E}[\bm{e}_k]&= (I-G)\mathbb{E}[\bm{e}_{k-1}]+G\bm{d}+B^{-1}\widetilde{A}^\top \mathbb{E}[H_k](\bm{\tilde b}-\widetilde{A}\bm {\tilde x}_{LS})\nonumber\\
&=(I-G)^k\bm{e}_{0}+\!\sum_{j=0}^{k-1}(I-G)^j G\bm{d}+\!\sum_{j=0}^{k-1}(I-G)^jB^{-1}\widetilde{A}^\top \mathbb{E}[H_k](\bm{\tilde b}-\widetilde{A}\bm {\tilde x}_{LS})\nonumber\\
&=(I-G)^k\bm{e}_{0}+\left[I-(I-G)^k\right]\bm{d}+\!\sum_{j=0}^{k-1}(I-G)^jB^{-1}\widetilde{A}^\top \mathbb{E}[H_k](\bm{\tilde b}-\widetilde{A}\bm {\tilde x}_{LS}).
\end{align}

Let $\bm{e}_0=(\bm{e}_0)_{\mathcal {R}(G)}+(\bm{e}_0)_{\mathcal {N}(G)}$
and
$\bm{d}=\bm{d}_{\mathcal {R}(G)}+\bm{d}_{\mathcal {N}(G)}$ be the orthogonal decompositions of $\bm{e}_0$ and $\bm{d}$,
respectively, with respect to the $B$-inner product. Let $\bm{\tilde r}=\bm{\tilde b}-\widetilde{A}\bm {\tilde x}_{LS}$. Taking the
$B$-norm on both sides of \eqref{516}, we obtain
\begin{align*}
\left\|\mathbb{E}[\bm{e}_k]\right\|_B&=\left\|(I-G)^k\bm{e}_{0}+\left[I-(I-G)^k\right]\bm{d}+\!\sum_{j=0}^{k-1}(I-G)^jB^{-1}\widetilde{A}^\top \mathbb{E}[H_k] \bm{\tilde r}\right\|_B\\
& \leq \left\|(I-G)^k\bm{e}_{0}\right\|_B+\left\|\left[I-(I-G)^k\right]\bm{d}\right\|_B+\left\|\!\sum_{j=0}^{k-1}(I-G)^jB^{-1}\widetilde{A}^\top \mathbb{E}[H_k] \bm{\tilde r}\right\|_B\\
&\leq \left\|(I-G)^k(\bm{e}_0)_{\mathcal {R}(G)}\right\|_B+\left\|(\bm{e}_0)_{\mathcal {N}(G)}\right\|_B+\left\|\left[I-(I-G)^k\right]\bm{d}_{\mathcal {R}(G)}\right\|_B\\
&\quad+\left\|\!\sum_{j=0}^{k-1}(I-G)^jB^{-1}\widetilde{A}^\top \mathbb{E}[H_k] \bm{\tilde r}\right\|_B\\
&\leq (1-\lambda_r)^{k}\left\|(\bm {\tilde x}_0-\bm {x}_{LS})_{\mathcal {R}(G)}\right\|_B+\big\|(\bm {\tilde x}_0-\bm {x}_{LS})_{\mathcal {N}(G)}\big\|_B\\
&\quad+\left[1-(1-\lambda_1)^k\right]\left\|(\bm {\tilde x}_{LS}-\bm {x}_{LS})_{\mathcal {R}(G)}\right\|_B+\frac{\|B^{-1}\widetilde{A}^\top \mathbb{E}[H_k] \bm {\tilde r}\|_B}{\lambda_r},
\end{align*}
where the last step follows from \eqref{b2}--\eqref{b4} and the fact that $B^{-1}\widetilde{A}^\top \mathbb{E}[H_k] \bm {\tilde r} \in\mathcal{R}(G)$.
\end{proof}

\begin{remark}
{\it 
Two remarks are in order. First, Theorems~\ref{th51} and~\ref{th52}
show that the sketch-and-project method applied to the doubly noisy
linear system \eqref{1.1} converges to a ball centered at
$\bm{x}_{\mathrm{LS}}$ without imposing any additional restrictions
on the initialization. Theorem~\ref{th51} establishes an upper bound
for
$\mathbb{E}\!\left[
\left\|\widetilde{\bm{x}}_k-\bm{x}_{\mathrm{LS}}\right\|_B^2
\right]
$
and shows that its convergence rate is governed by the smallest
nonzero eigenvalue $\lambda_r$ of $G$. In contrast,
Theorem~\ref{th52} establishes an upper bound for
$
\left\|
\mathbb{E}\!\left[
\widetilde{\bm{x}}_k-\bm{x}_{\mathrm{LS}}
\right]
\right\|_B
$
and shows that, during the early iterations, its convergence behavior
is influenced not only by $\lambda_r$ but also by the largest
eigenvalue $\lambda_1$ of $G$. However, as the iteration count $k$
increases, the asymptotic convergence rates associated with both
bounds are governed solely by $\lambda_r$.
Second, the convergence horizon in Theorem~\ref{th51} depends on
$\|\left(\widetilde{\bm{x}}_0-\bm{x}_{\mathrm{LS}}\right)_{\mathcal{N}(G)}\|_B^2$, 
$\lambda_r$,
and
$\bm{r}^{\top}\mathbb{E}[H_k]\bm{r}$.
By comparison, the convergence horizon in Theorem~\ref{th52}
depends on
$\lambda_r$,
$\left\|B^{-1}\widetilde{A}^{\top}\mathbb{E}[H_k]\widetilde{\bm{r}}\right\|_B$,
$\left\|\left(\widetilde{\bm{x}}_0-\bm{x}_{\mathrm{LS}}\right)_{\mathcal{N}(G)}\right\|_B$,
and
$\left\|\left(\widetilde{\bm{x}}_{\mathrm{LS}}-\bm{x}_{\mathrm{LS}}\right)_{\mathcal{R}(G)}\right\|_B$.
}
\end{remark}

\begin{remark}
{\it Theorems \ref{th51} and \ref{th52} apply to both the full column rank and rank deficient cases. Note that RK and RBK can be viewed as special cases of the sketch-and-project method. Therefore, Theorem \ref{th51} overcomes the limitation of Theorem \ref{thm3.5} and yields a convergence rate that is independent of $\xi_{k-1}$. Moreover, by employing the analytical technique developed in Theorem \ref{th52}, the full row rank assumption on $\widetilde{A}_{\tau}$ imposed in Theorems \ref{thm3.3} and \ref{thm3.4} can be removed. Taken together, these results analyze, from different perspectives, the factors affecting the upper bound on
$\left\|\mathbb{E}\!\left[\bm{x}_k-\bm{x}_{\mathrm{LS}}\right]\right\|_2$ for RBK and RDBK solving \eqref{1.1}.

}

\end{remark}

\section{Numerical experiments}\label{sec5}
In this section, we perform some numerical experiments to validate our theoretical findings.
Four methods, including RK \cite{Strohmer-2009}, REK \cite{Zouzias-2013}, RBK \cite{Need-2014}, and RDBK \cite{Need-2015}, are run on doubly noisy linear systems, with the initial value $\bm{\tilde{x}}_0=\bm{0}$. 

In all the experiments, both the coefficient matrix $A$ and measurement vector $\bm b$ are contaminated by additive noise.
More precisely,  we construct the noisy data as $\widetilde{A}=A+E$ and $\bm {\tilde{b}}=\bm b+\bm f$, where $E$ and $\bm f$ are additive noises with entries generated from a standard normal distribution. Moreover, the relative noise levels satisfy $\left\| E\right\|_F / \left\|A\right\|_F=\delta_A$ and $\left\|\bm f\right\|_2 / \left\|\bm b\right\|_2=\delta_b$, where $\delta_A>0$, $\delta_b>0$ are the given the noise magnitudes: 
$$
E = \mathtt{randn(m, n)},~~E = \mathtt{\delta_A*norm(A,\textquotesingle fro\textquotesingle)*E/norm(E,\textquotesingle fro\textquotesingle)},
$$
$$
\bm f = \mathtt{randn(m,1)},~~~\bm f = \mathtt{\delta_b*norm(\bm{b},2)*\bm {f}/norm(\bm f,2)}.
$$
All the numerical experiments are conducted in \textsc{MATLAB} (R2023b) on a Windows 10 64-bit workstation, equipped with a 32-core 2.20 GHz Intel Platinum 8253 CPU and 256 GB of RAM. All the numerical results are averaged over 10 independent trials.

\subsection{A comparison of RK and REK for doubly noisy linear systems}\label{sec5.2}
In this section, we compare the practical results $\mathbb{E}\|\bm{\tilde{x}}_k-\bm{x}\|_2$ and $\|\mathbb{E}[\bm{\tilde{x}}_k-\bm{x}]\|_2$ obtained from the RK and REK methods for doubly noisy linear systems, and the theoretical bounds provided by Theorem \ref{thm3.1} and Theorem \ref{thm3.2}.
In this experiment, the coefficient matrix $A$ is from the University of Florida Sparse Matrix Collection\footnote{https://sparse.tamu.edu/}. 
Table \ref*{test} summarizes the detailed information on these real-world datasets.
	\begin{table}[ht]
	\centering
	\caption{\it Detailed information about the test matrices}
	\label{test}
	\begin{tabular}{|l|c|c|c|}
		\hline
		Matrix & $m \times n$ & \ Nonzeros & Background \\
		\hline
		cari & 400 $\times$ 1200 & 152800  & Linear programming problem \\ \hline
		$lp\_ maros\_r7$ & 3136 $\times$ 9408 & 144848 & Linear programming problem \\ \hline
		df2177 & 630$\times$ 10358 & 22336  & Linear programming problem \\ \hline
       crew1 & 135$\times$ 6469 & 46950 & Linear programming problem \\ \hline
	\end{tabular}
\end{table}

The solution $\bm{x}$ of the underlying noiseless system \eqref{1.2} and the right-hand side vector $\bm{b}$ are generated by
\[
\bm{x}=\mathtt{randn(n,1)}, \qquad \bm{b}=A\bm{x},
\]
and the noise levels are set to be $\delta_A=\delta_b=0.05$. Figure \ref{fig3} plots the results.


Some remarks are in order. First, it is seen from the curves of $\mathbb{E}\|\bm{\tilde{x}}_k-\bm{x}_{LS}\|_2$ and $\|\mathbb{E}[\bm{\tilde{x}}_k-\bm{x}_{LS}]\|_2$ in Figure \ref{fig3} that, the RK method
converges faster than the REK method, while the convergence horizons of the two methods
coincide with each other; see Remark \ref{rema3.1}. 
Second, the theoretical bounds produced by the RK and REK methods also exhibit these two characteristics. Asymptotically, as the number of iterations increase, the theoretical bounds given by Theorem \ref{thm3.2} and Theorem \ref{thm3.1} and the practical bounds are consistent.

Third, the upper bound in Theorem \ref{thm3.2} initially increases and then decreases after reaching its maximum value. Indeed, it is from the combined effects of the two terms
\[
\alpha^k=\left(1-\frac{\widetilde{\sigma}_{\min}^2}{\|\widetilde{A}\|_F^2}\right)^k
\quad\text{and}\quad
\beta=1-\left(1-\frac{\widetilde{\sigma}_{\max}^2}{\|\widetilde{A}\|_F^2}\right)^k
\]
in the theoretical bound. During the early iterations, the increase induced by $\beta$ dominates the decrease induced by $\alpha^k$, causing the theoretical bound to increase. As the iteration number $k$ increases, the increment of $\beta$ gradually approaches one, whereas $\alpha^k$ continues to decay, and hence the theoretical bound eventually decreases. Moreover, our theory is very effective in characterizing the asymptotic behavior of the REK method.

\begin{figure}[htbp]
	\centering
	\includegraphics[width=6cm,height=4.5cm]{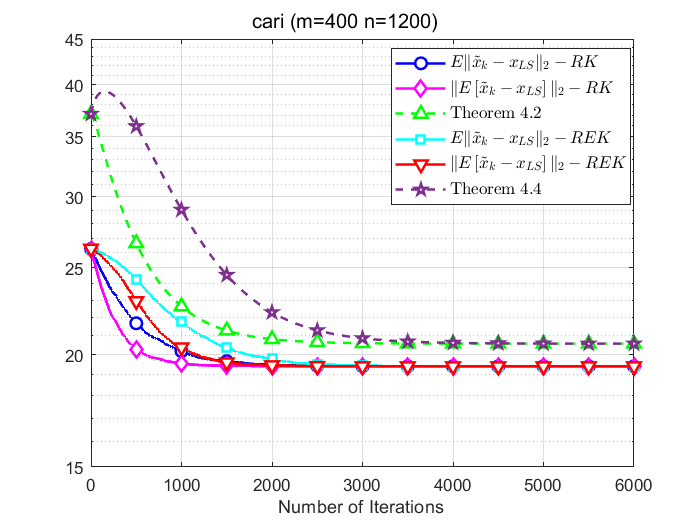}
	\includegraphics[width=6cm,height=4.5cm]{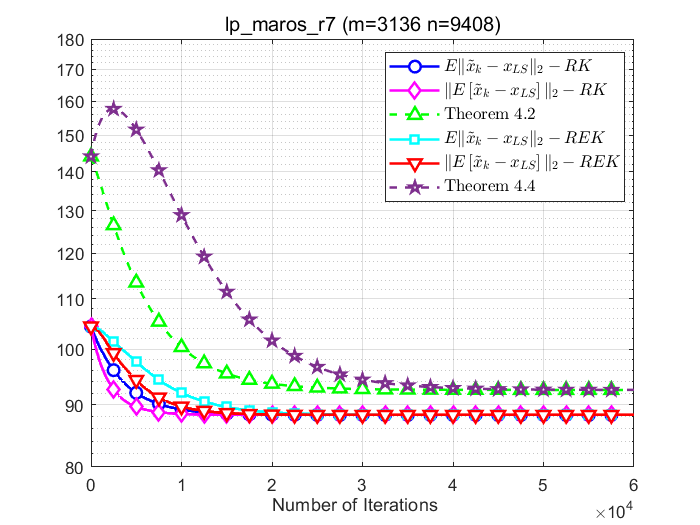}
	\includegraphics[width=6cm,height=4.5cm]{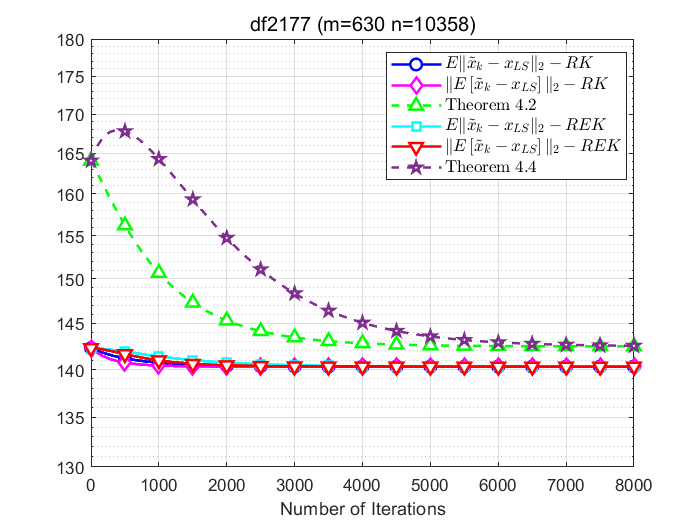}
	\includegraphics[width=6cm,height=4.5cm]{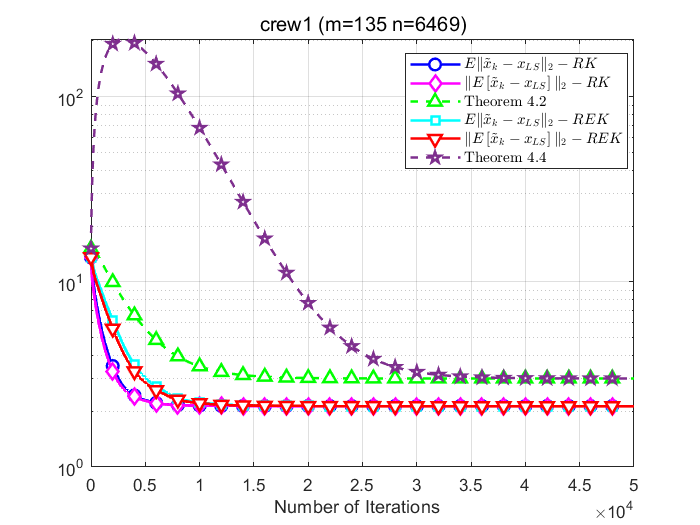}
	\caption{Example \ref{sec5.2}: {\it Comparison of the practical results $\mathbb{E}\|\bm{\tilde{x}}_k-\bm{x}_{LS}\|_2$ and $\|\mathbb{E}[\bm{\tilde{x}}_k-\bm{x}_{LS}]\|_2$ with the theoretical bounds established in Theorems \ref{thm3.1} and \ref{thm3.2} for the RK and REK methods, respectively. First row, left to right: \texttt{cari}, \texttt{df2177}; second row, left to right: \texttt{lp\_ maros\_r7}, and \texttt{crew1} datasets.} }\label{fig3}

\end{figure}

\subsection{A comparison of RBK and RDBK for doubly noisy linear systems}\label{sec5.4}
In this section, we compare the practical results $\mathbb{E}\|\bm{\tilde{x}}_k-\bm{x}\|_2$ and $\|\mathbb{E}[\bm{\tilde{x}}_k-\bm{x}]\|_2$ obtained from the RBK and RDBK methods for doubly noisy linear systems, and the theoretical bounds provided by Theorem \ref{thm3.3} and \ref{thm3.4}. In this experiment, the coefficient matrix $A$ is generated from the standard Gaussian distribution:
$$
A ={\tt randn(m, n)},
$$
with $m=20000, n=5000$. 

The solution $\bm{x}$ of the underlying noiseless system \eqref{1.2} and the right-hand side vector $\bm{b}$ are generated by
\[
\bm{x}=\mathtt{randn(n,1)}, \qquad \bm{b}=A\bm{x}.
\]
In this experiment, the noise magnitudes are set to be $\delta_A=0.01$ and $\delta_b=0.01$. The row index set $\{1,2,\ldots,m\}$ is partitioned into 30, 40, 50, and 60 blocks. Similarly, the column index set $\{1,2,\ldots,n\}$ is also partitioned into 30, 40, 50, and 60 blocks. Figure \ref{fig6} plots the results.

Some remarks are in order. First, it is seen from the curves of $\mathbb{E}\|\bm{\tilde{x}}_k-\bm{x}_{LS}\|_2$ and $\|\mathbb{E}[\bm{\tilde{x}}_k-\bm{x}_{LS}]\|_2$ in Figure \ref{fig6} that, the RBK method
converges faster than the RDBK method, while the convergence horizons of the two methods
are comparable. 
Second, the theoretical results also indicate that RBK is superior to RDBK; see Remark \ref{rema3.2}. 

Third, the numerical results in Figures \ref{fig6} show that increasing the row partition number $p$ enlarges the convergence horizon and makes the theoretical bounds less sharp. Meanwhile, the two algorithms converge more slowly. This phenomenon can be explained by the fact that, as $p$ decreases, each block contains more rows, thereby potentially increasing the dimension of the row space of each row submatrix $\widetilde A_{\tau_i}$. Consequently, the projection matrix
$V_{r_i}V_{r_i}^{\top}$
spans a larger subspace and captures more directions. Therefore, the quantity 
$$
\eta=\frac{1}{\|\tilde{A} \|_F}\lambda_{\min}(\sum_{i=1}^{p}\|\widetilde{A}_{\tau_i}\|_FV_{r_i}V_{r_i}^\top)
$$ increases, which yields sharper theoretical bounds in Theorems \ref{thm3.3} and \ref{thm3.4} and faster convergence.

\begin{figure}[htbp]
\centering

\begin{subfigure}{0.48\textwidth}
\centering
\includegraphics[width=\linewidth,height=4.8cm]{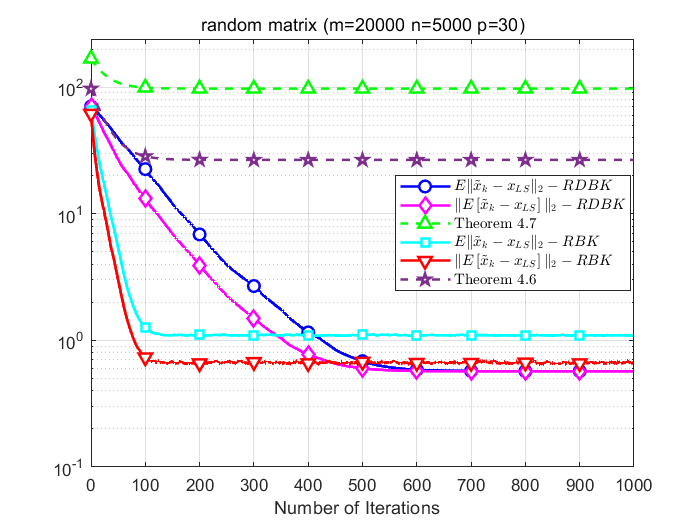}
\end{subfigure}
\begin{subfigure}{0.48\textwidth}
\centering
\includegraphics[width=\linewidth,height=4.8cm]{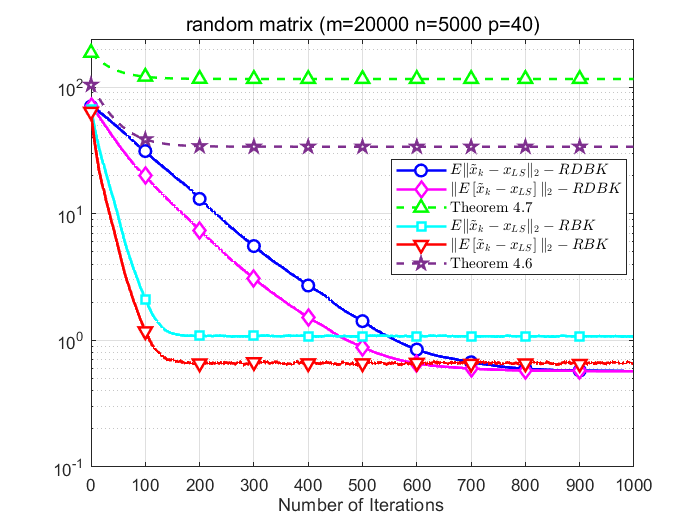}
\end{subfigure}

\vspace{0.2cm}

\begin{subfigure}{0.48\textwidth}
\centering
\includegraphics[width=\linewidth,height=4.8cm]{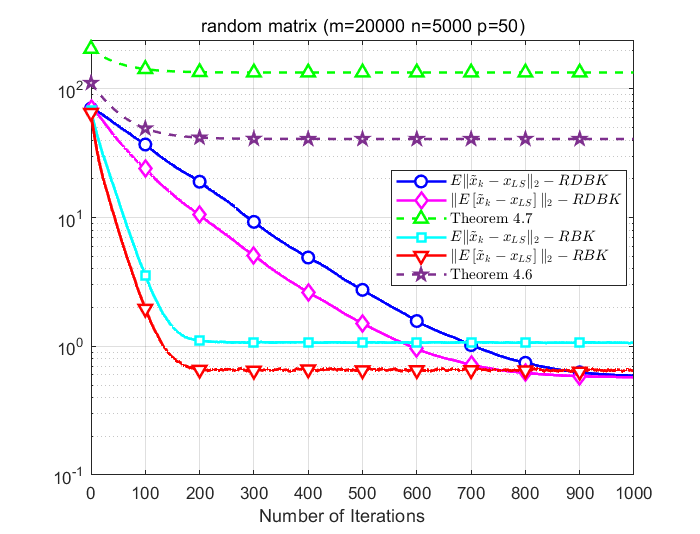}
\end{subfigure}
\begin{subfigure}{0.48\textwidth}
\centering
\includegraphics[width=\linewidth,height=4.8cm]{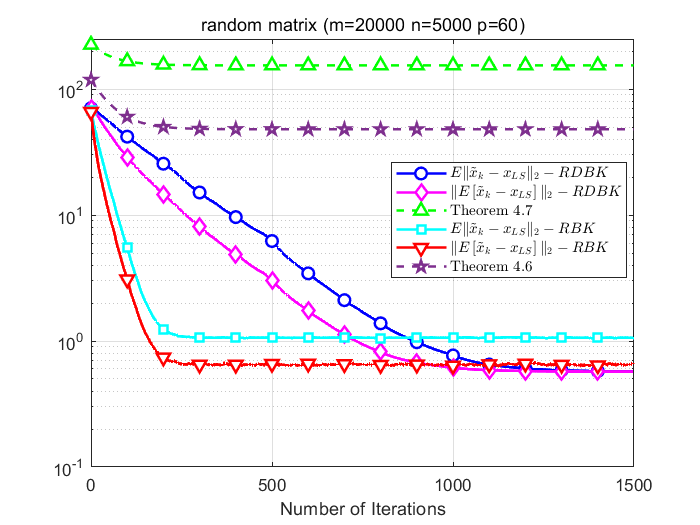}
\end{subfigure}

\caption{Example \ref{sec5.4}: {\it Comparison of the practical results $\mathbb{E}\|\bm{\tilde{x}}_k-\bm{x}_{LS}\|_2$ and $\|\mathbb{E}[\bm{\tilde{x}}_k-\bm{x}_{LS}]\|_2$ with the theoretical bounds established in Theorems \ref{thm3.3} and \ref{thm3.4} for the RBK and RDBK methods, respectively. The noisy coefficient matrix is $\tilde{A}=A+E$ with $A=\mathtt{randn(m, n)}$ ($m=20000,~n=5000$).} 
}\label{fig6}

\end{figure}

\section{Concluding remarks}

In this work, we focus on the convergence behavior of some popular randomized Kaczmarz-type methods, including RK, REK, RBK and RDBK, for solving doubly noisy linear systems. Different from the existing works \cite{Bergou-2024,Bergou-2025}, we are interested in the distance between the approximations obtained from the Kaczmarz-type methods
and the LS solution of the underlying noiseless linear system \eqref{1.2}, and our results are free of the assumption $\bm {\tilde x}_{0}- \bm {x}_{LS} \in \mathcal{R}(\widetilde{A}^\top)$.
Moreover, we extend our analysis to the more general sketch-and-project framework, and the convergence bounds are derived from the perspectives of both the expectation of the error norm and the norm of the expected error.


Our results identify the key factors governing the convergence rates and error bounds of these methods and indicate that, for doubly noisy linear systems, RK and RBK can outperform REK and RDBK, respectively, in terms of the asymptotic convergence rate or the convergence horizon. The analysis covers both rank-deficient and full-rank problems, as well as both consistent and inconsistent linear systems. Numerical experiments validate the theoretical findings and demonstrate the effectiveness of the derived bounds.

In future work, it is interesting to design efficient methods that can achieve a faster convergence rate and a lower convergence horizon for doubly noisy linear systems. Moreover, how to improve the theoretical bound given in Theorem \ref{thm3.2}, such that the phenomenon of sudden increase in the early stage can be eliminated deserves further investigation.

\section*{Acknowledgments}
We are grateful to Professor Ning Zheng for helpful discussions. 

\bibliographystyle{siamplain}
\bibliography{references}
\end{document}